\documentclass[12pt]{amsart}
\usepackage[margin=1in,twoside=false]{geometry}\headsep = 0.5 in

\usepackage{amsmath,amssymb,amsthm,latexsym,graphicx,verbatim, enumerate}
\usepackage[mathrefs,edtable, displaymath]{lineno}
\usepackage{setspace}
\usepackage{multirow}
\usepackage{subfig,
keyval,cancel,caption}
\usepackage[normalem]{ulem}
\usepackage[
colorlinks=true, backref=page
]{hyperref}
\usepackage[]{floatrow}
\usepackage{tikz}
\usetikzlibrary{calc}
\usetikzlibrary{backgrounds}
\usetikzlibrary{shapes}
\usetikzlibrary{positioning}
\usetikzlibrary{decorations, decorations.pathmorphing, decorations.text}
\usetikzlibrary{snakes}

\newtheorem{theorem}{Theorem}[section]
\newtheorem{corollary}[theorem]{Corollary}
\newtheorem{conjecture}[theorem]{Conjecture}
\newtheorem{lemma}[theorem]{Lemma}
\newtheorem{proposition}[theorem]{Proposition}

\newtheorem{question}[theorem]{Question}

\newtheorem{observation}[theorem]{Observation}

\theoremstyle{remark}

\newcommand{\RR}{{\mathbb R}}

\newcommand{\bc}{\begin{center}}
\newcommand{\ec}{\end{center}}
\newcommand{\be}{\begin{enumerate}}
\newcommand{\ee}{\end{enumerate}}
\newcommand{\bi}{\begin{itemize}}
\newcommand{\ei}{\end{itemize}}

\newcommand{\bv}[1]{\ensuremath{ \vec{\mathbf{#1}}} }

\usepackage[]{color}
\newcommand{\blue}[1]{{\color{blue}#1}}

\DeclareMathOperator*{\cupdotop}{\ensuremath{\dot{\cup}}}
\DeclareMathOperator{\obs}{\ensuremath{\mathrm{obs}}}
\DeclareMathOperator{\obsout}{\ensuremath{\obs_\mathrm{out}}}

\title{Graphs with Obstacle Number Greater than One}

\author{Leah Wrenn Berman}
\author{Glenn G.~Chappell}
\author{Jill R.~Faudree}
\author{John Gimbel}\thanks{John Gimbel's research was partially supported by the Czech Ministry of Education Grant \#ERC-CZ 1201}
\author{Chris Hartman}
\author{Gordon I.~Williams}

\begin{document}


\begin{abstract}
An \emph{obstacle representation} of a graph $G$
is a straight-line drawing of $G$ in the plane
together with a collection of connected subsets of the plane,
called \emph{obstacles}, that block all non-edges
of $G$ while not blocking any of the edges of $G$.
The \emph{obstacle number} $\obs(G)$ is
the minimum number of obstacles required to represent $G$.

We study the structure of graphs
with obstacle number greater than one.
We show that the icosahedron has obstacle number $2$,
thus answering a question of
Alpert, Koch, \& Laison asking whether all planar graphs
have obstacle number at most $1$.
We also show that the $1$-skeleton of
a related polyhedron, the \emph{gyroelongated $4$-bipyramid},
has obstacle number $2$.
The order of this graph is $10$,
which is also the order of the smallest known graph with
obstacle number $2$.

Some of our methods involve instances of the Satisfiability problem;
we make use of various ``SAT solvers'' in order to produce
computer-assisted proofs.
\end{abstract}

\maketitle


\section{Introduction}

All graphs will be finite, simple, and undirected.
Following Alpert, Koch, \& Laison~\cite{AlpKocLai10}, we define an
 \emph{obstacle representation} of a graph $G$ to be
 a straight-line drawing of $G$ in the plane,
together with a collection of connected subsets of the plane,
called \emph{obstacles},
such that no obstacle meets the drawing of $G$,
while every non-edge of $G$ intersects at least one obstacle. By \emph{non-edge}, we mean a pair of distinct vertices $a,b$ of $G$ where $ab$ is not an edge of $G$.
The least number of obstacles required to represent $G$
is the \emph{obstacle number} of $G$,
denoted $\obs(G)$. For clarity, we will sometimes refer to this as the \emph{ordinary} obstacle number. In an obstacle representation of a graph $G$,
an \emph{outside obstacle} is an obstacle
that is contained in the unbounded component of
the complement of the drawing of $G$.
Any other obstacle is
an \emph{interior obstacle}. Note that by \cite{JohSar11}, determining whether $\obs(G)\le k$ for a given $k$ is \emph{not} in NP.

Figure~\ref{F:ObsEx} shows several obstacle representations
for the Petersen graph $P$.
The usual drawing of $P$ requires five obstacles.
This is illustrated in Figure~\ref{Fp:PetersenAny}
(edges of $P$ are shown as thick line segments,
while \emph{non-edges} are shown as thin line segments).

An obstacle may be any connected plane region;
however, it is convenient to expand each obstacle until it is as large
as possible.
Each of the resulting obstacles is an open plane region
forming a connected component of the complement of the drawing of the graph.
The boundary of the obstacle---which
is not part of the obstacle itself---is
formed by appropriate parts of the edges of the graph.
Figure~\ref{Fp:PetersenFull} illustrates such obstacles
for $P$.

However, $P$ can be represented using fewer obstacles;
$\obs(P)$ is in fact $1$.
Figure~\ref{Fp:PetersenObs} shows an obstacle representation of $P$
using only a single obstacle
(adapted from Laison \cite{Lai}).

\begin{figure}[htbp]
\begin{center}
\ffigbox[\FBwidth]{
\begin{subfloatrow*}[4]
\subfloat[]{\label{Fp:PetersenAny}\includegraphics[width=.3\linewidth]{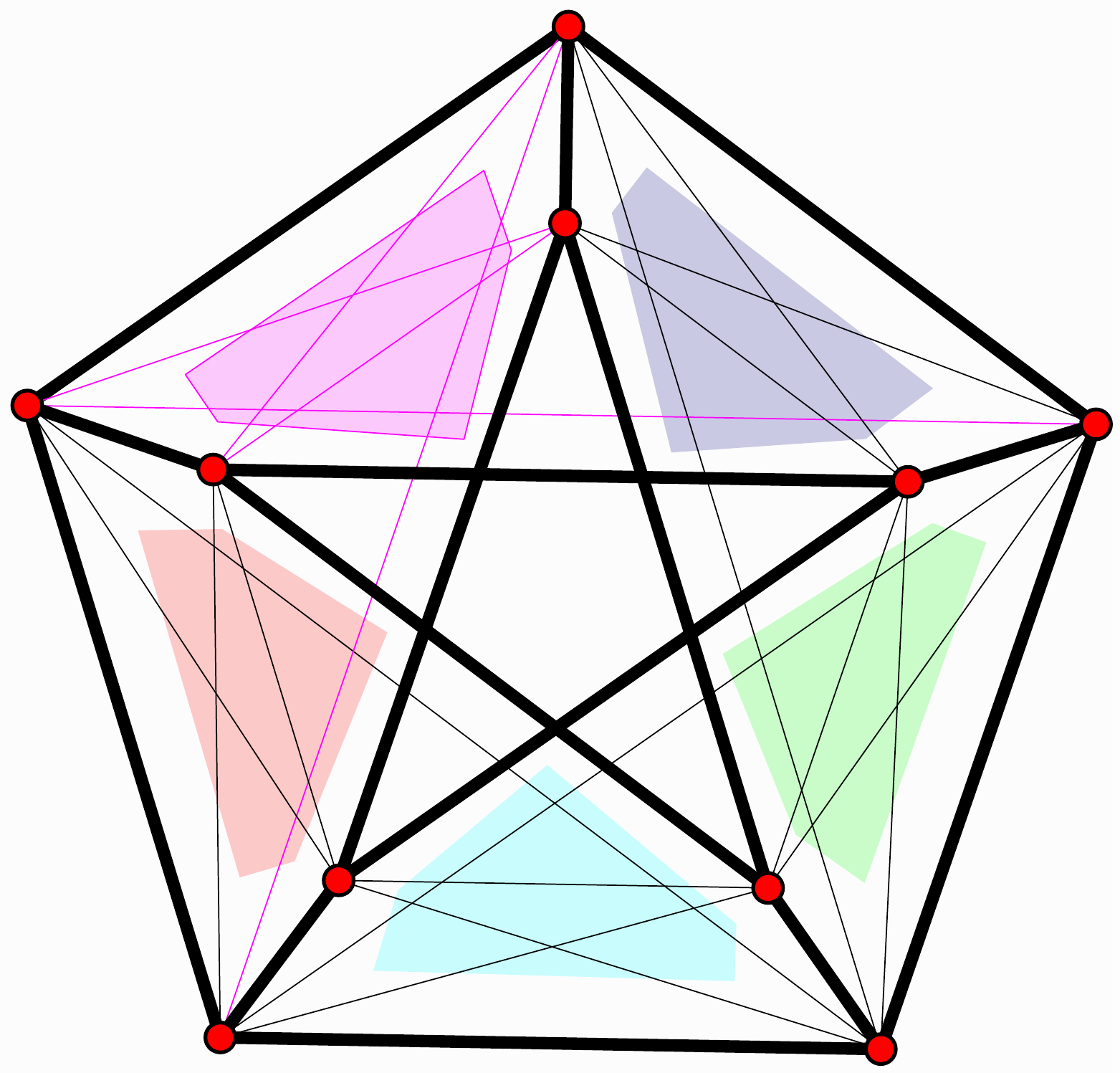}}
\subfloat[]{\label{Fp:PetersenFull}\includegraphics[width=.3\linewidth]{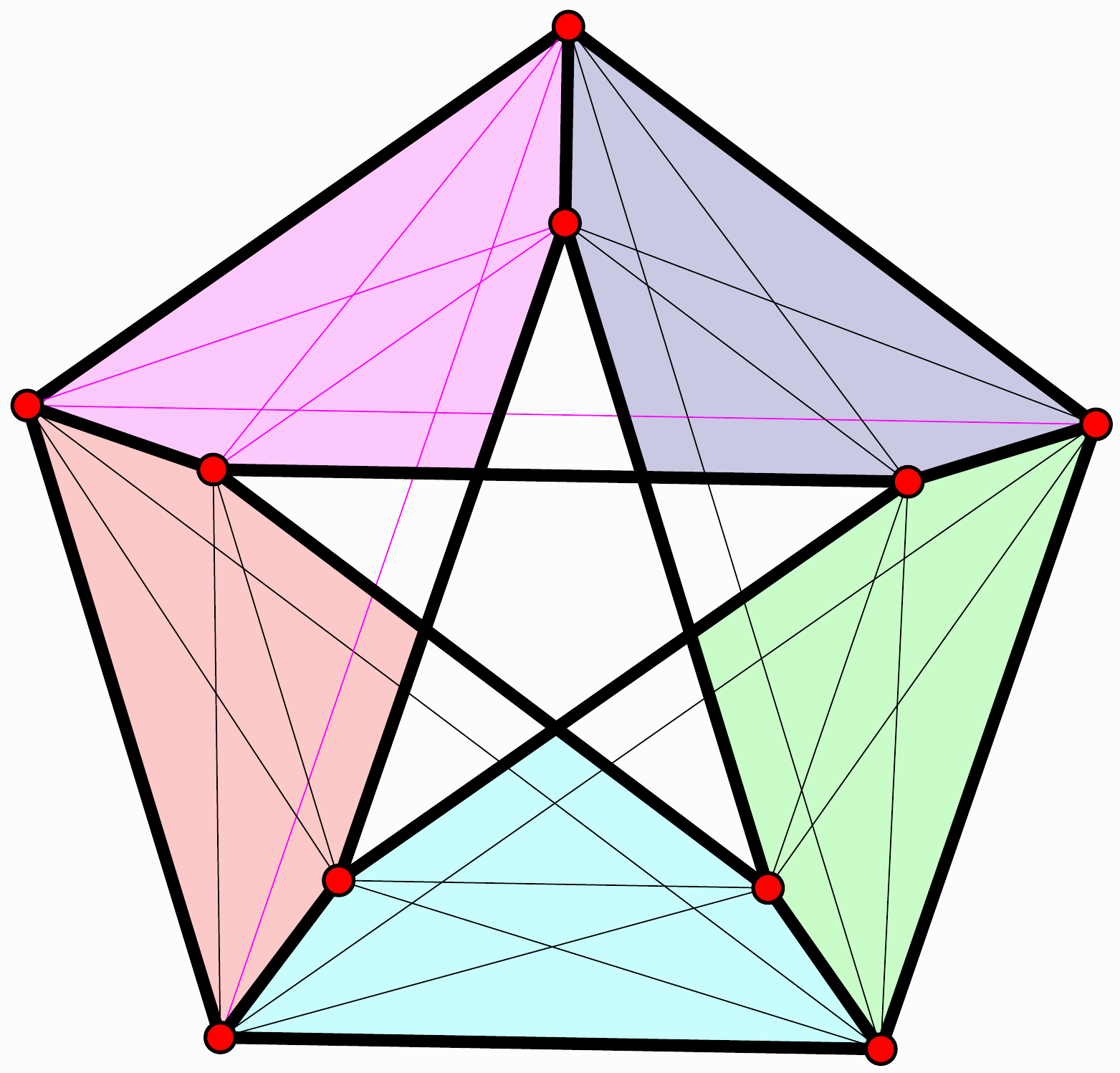}}
\subfloat[]{\label{Fp:PetersenObs}\includegraphics[width=.3\linewidth]{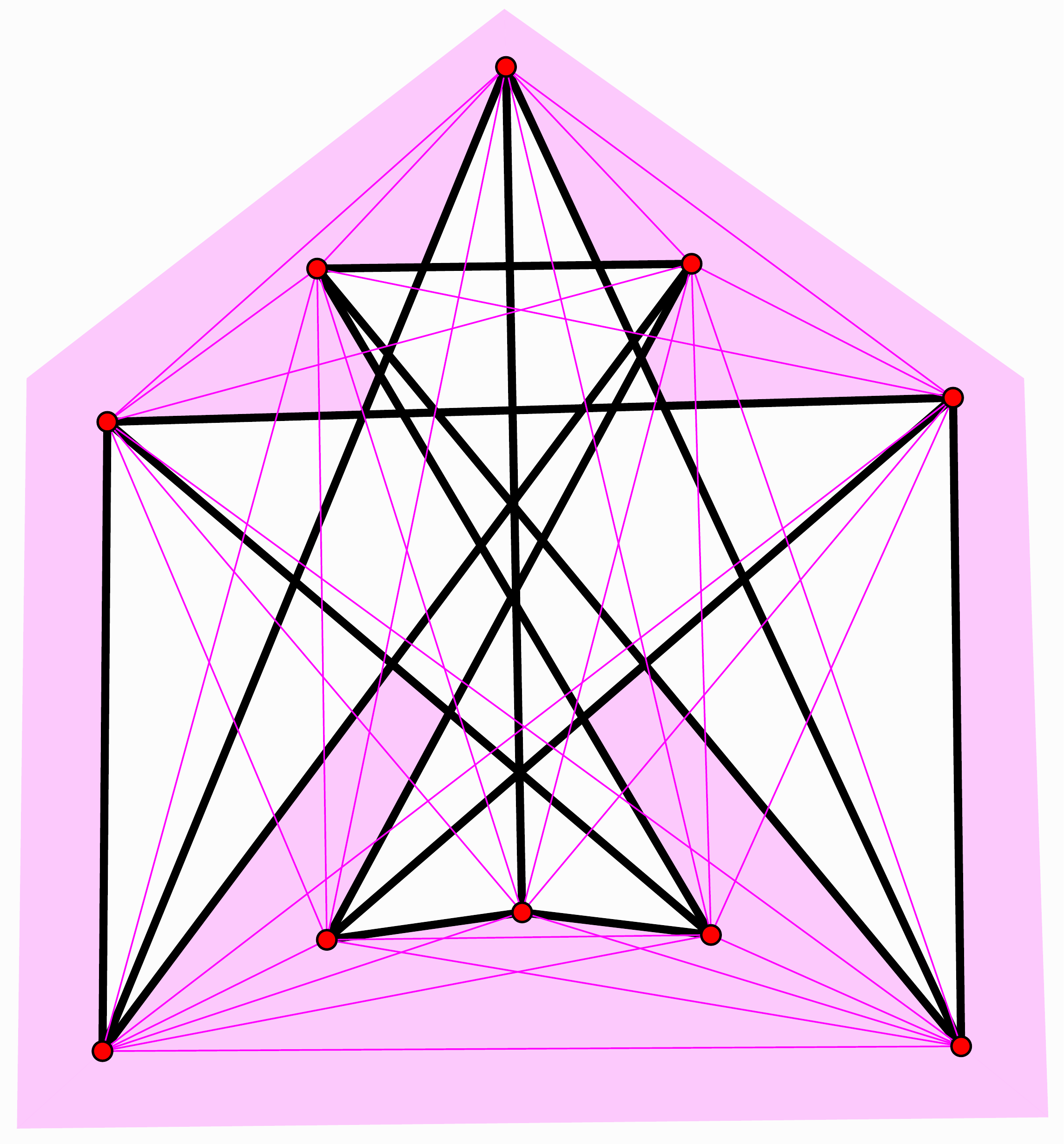}}
\end{subfloatrow*}}{
\caption{(A) The familiar embedding of the Petersen graph,
which requires five obstacles;
the non-edges blocked by the pink obstacle are colored pink.
(B) Here, we expand the obstacles as much as possible,
until each is a maximal connected region in the complement of the drawing.
(C) However, $\obs(P)$ = 1, using a single outside obstacle.}
\label{F:ObsEx}}
\end{center}
\end{figure}

A related parameter,
in which each obstacle is a single point,
was studied by Matou\v{s}ek~\cite{Mat09}
and by Dumitrescu, Pach, and T\'{o}th~\cite{DumPacTot09}.
The study of the obstacle number \textit{per se}
was initiated by Alpert, Koch, \& Laison~\cite{AlpKocLai10}.
These parameters
have since been investigated by others~\cite{FulSaeSar2013,JohSar11,MukPacPal12,MukPacSar10,PacSar10,PacSar11}.

Alpert, Koch, \& Laison~\cite[Thm.~2]{AlpKocLai10}
showed that there exist graphs with arbitrarily high obstacle number
and asked~\cite[p.~229]{AlpKocLai10}
for the smallest order of a graph with obstacle number greater than $1$.
They proved~\cite[Thm.~4]{AlpKocLai10}
that the graph $K^{*}_{5,7}$ has obstacle number $2$,
where $K^{*}_{a,b}$ (with $a\le b$) denotes the graph obtained
by removing a matching of size $a$
from the complete bipartite graph $K_{a,b}$.
Pach \& Sari\"{o}z~\cite[Thm.~2.1]{PacSar11}
found a smaller example of a graph with obstacle number $2$.

\begin{theorem}[Pach \& Sari\"{o}z] \label{T:kstar55}
$\obs(K^*_{5,5}) = 2$.
\end{theorem}
The Pach-Sari\"{o}z example, of order $10$,
is apparently the smallest known graph with obstacle number greater than $1$.
In section~\ref{S:small} we consider
the question of whether any graph of smaller order
has obstacle number greater than $1$.
We do not answer this question;
however, in Proposition~\ref{P:gesdp-obs} we will provide
an example of a \emph{planar} graph of order $10$ with obstacle number $2$.

We define the \emph{outside obstacle number} of $G$
to be the least number of obstacles required to represent $G$,
such that one of the obstacles is an outside obstacle---or
zero if $G$ has obstacle number zero.
We denote the outside obstacle number of $G$ by
$\obsout(G)$.

Clearly we have
$\obs(G)\le \obsout(G)\le \obs(G)+1$
for every graph $G$.
Alpert, Koch, \& Laison~\cite[p.~229]{AlpKocLai10}
asked (using different terminology) whether
every graph $G$ with
$\obs(G) = 1$
also has
$\obsout(G) = 1$.
We ask a more general question.

\begin{question}
Is it true that
$\obs(G) = \obsout(G)$
for every graph $G$?\end{question}

Alpert, Koch, \& Laison~\cite[p.~231]{AlpKocLai10}
asked whether every planar graph has obstacle number at most $1$
(also see a series of questions in the Open Problem Garden \cite{Gar}).
They further asked for the obstacle numbers of
the icosahedron and the dodecahedron.
In Sections~\ref{S:obsout}, \ref{S:obs1}, and \ref{S:obs2}
we develop tools for determining the obstacle numbers of
particular graphs,
and we use them to address these questions.

In Section~\ref{S:obsout} we consider the outside obstacle number.
We make use of instances of the Satisfiability problem (SAT)
to produce computer-assisted proofs that
the outside obstacle number of certain graphs is at least $2$.

In Section~\ref{S:obs1} we consider the ordinary obstacle number.
We show that a lower bound on the outside obstacle number of a graph
implies a lower bound on the obstacle number of a different graph.
We use this to answer the first of the above-mentioned questions
by constructing a planar graph with obstacle number $2$.

In Section~\ref{S:obs2} we develop methods similar to
those of Section~\ref{S:obsout}
for producing computer-assisted proofs that the ordinary
obstacle number of certain graphs is at least $2$.
We answer another of the above-mentioned questions by showing
that the obstacle number of the icosahedron is $2$.
We also show that the obstacle number of the dodecahedron is
$1$.
We further describe a graph of order $10$
that has obstacle number $2$.
This graph, which we call $X_4$, has the same order as the
Pach-Sari\"{o}z example,
but, unlike that graph,
it is planar.

We conclude this section with some easy observations,
which we will make use of throughout the remainder of this paper.

\begin{observation} \label{O:perturb}
Given an obstacle representation of a graph $G$,
we can perturb all vertices an arbitrarily small distance
to obtain an essentially equivalent obstacle representation
in which no three vertices are collinear.
\end{observation}

Because of the above observation,
we will generally assume that our obstacle representations
have the property that no three vertices are collinear. Lest there be any confusion, only obstacles block edges, vertices do not, so this has no impact on the definition of obstacle number.
\begin{observation}\label{O:simplyconnected} Obstacles
are not required
to be simply connected,
but adding this requirement
does not change the obstacle number.\end{observation}
To see why this is true we note two facts: (1) whether an edge is blocked or not by an obstacle is determined only by whether the edge somewhere crosses a portion of the boundary of the obstacle; (2) if an obstacle is not simply connected, 
then
it may be replaced with a simply connected obstacle that blocks those same edges by removing sufficiently small open sets about curves connecting portions of the complement that do not intersect the points on the boundary where edges of the graph cross the obstacle (the simplest case is cutting a small wedge out of an annulus blocking an edge connecting a vertex inside the annulus to a vertex outside the annulus). 

\begin{observation} \label{O:induced}
Let $G$ be a graph,
and let $H$ be an induced subgraph of $G$.
Given an obstacle representation of $G$,
removing all vertices of $G$ that are not in $H$,
and keeping the same obstacles, results
in an obstacle representation of $H$.\end{observation}

The next result follows easily from the above observation.

\begin{proposition} \label{P:closed-induced}
Let $k$ be a nonnegative integer.
The class of all graphs $G$ such that $\obs(G)\le k$
is closed under taking induced subgraphs,
and similarly for
$\obsout$.
\end{proposition}

Proposition~\ref{P:closed-induced}
does not hold if the word ``induced'' is removed,
as the classes in question are generally not closed under
taking arbitrary subgraphs.
For example, Alpert, Koch, \& Laison~\cite[Thm.~2]{AlpKocLai10}
showed that there are graphs with arbitrarily large obstacle number.
But every graph is a subgraph of a complete graph,
and $\obs(K_n) = 0$ for every $n$.

\section{Small Graphs} \label{S:small}

As noted in the previous section,
Pach \& Sari\"{o}z showed that
$\obs(K^{*}_{5,5}) = 2$
(see Theorem~\ref{T:kstar55}).
What can we say about graphs of small order
with obstacle number greater than $1$?

\begin{proposition} \label{P:pendant}
Let $G$ be a graph.
Let $H$ be a graph obtained by starting with
$G$ and adding a new vertex of degree at most $1$.
If $G$ is complete and $H$ is not complete,
then
$\obs(H) = 1 = \obs(G)+1$.
Otherwise,
$\obs(H) = \obs(G)$.
Furthermore,
these continue to hold if
$\obs$
is replaced by
$\obsout$.\end{proposition}

\begin{proof}
We give the proof for $\obs$;
the proof for $\obsout$ is essentially the same.

We have $\obs(G) = 0$ if and only if $G$ is complete;
the cases when $G$ is complete follow easily.
Suppose that $G$ is not complete,
and let $v$ be the new vertex in $H$.

Begin with an obstacle representation of $G$.
First suppose that the degree of $v$ in $H$ is $0$.
In this case, Let $T$ be any obstacle of $G$.
Let $p$ be a point in the interior of $T$, replace $T$ by $T-p$,
and place vertex $v$ at the point $p$.
Then $\obs(H) = \obs(G)$
since the edges between $v$ and the vertices of
$G$ are blocked by $T$.

If instead $v$ has degree $1$,
then let $q\in G$ be the neighbor of $v$ in $H$.
Choose a point $p$ in the interior of an obstacle $T$
so that the line segment $S=pq$ intersects no other obstacle of $G$;
note that $S$ may cross edges of $G$.
(To find such a $T$ choose any obstacle $R$ and a point $r$ on its interior,
and then let $T$ be the first obstacle
crossed by the ray $\overrightarrow{qr}$.)
Now replace $T$ by $T-S$,
and place vertex $v$ at the point $p$.
(See Figure~\ref{F:addpendant} for an illustration.)
This results in an obstacle representation of $H$
using the same number of obstacles as that for $G$,
and also using an outside obstacle
if the representation of $G$ used one.\end{proof}

\begin{figure}[htbp]
\begin{center}
\includegraphics[]{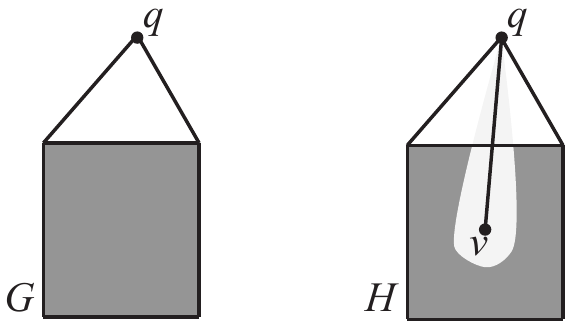}
\caption{The graph $H$ obtained from $G$ by adding a single pendant vertex.
The dark gray shaded area corresponds to an obstacle for each graph.}
\label{F:addpendant}
\end{center}
\end{figure}

We restate  \cite[Thm. ~5]{AlpKocLai10} as follows.
\begin{theorem}\label{l:AKL} The obstacle number of an outerplanar graph is at most 1.\end{theorem}

\begin{corollary} \label{P:cyc3}
Let $G$ be a graph with the property that
every cycle in $G$ has length $3$.
Then
$\obsout(G) \le 1$.\end{corollary}

\begin{proof}
Graph $G$ can have no minor isomorphic to
$K_4$ or $K_{2,3}$,
since these graphs both have $4$-cycles.
Hence $G$ is outerplanar,
so by Theorem \ref{l:AKL}, the obstacle number is at most $1$.
In their proof,
Alpert, Koch, \& Laison
construct an obstacle representation
using a single outside obstacle,
thus showing that $\obsout$ is at most $1$.
The result follows.\end{proof}

\begin{proposition} \label{P:ord5}
If a graph $G$ has order at most $5$,
then
$\obsout(G) \le 1$.\end{proposition}

\begin{proof}
By Proposition~\ref{P:pendant},
we may assume that $G$ has minimum degree at least $2$,
since vertices of degree at most $1$ do not affect the obstacle number.
There are only four graphs of order at most $4$
whose minimum degree is at least $2$:
$K_3$, $C_4$, $K_4$, and $K_4-e$.
All of these are easily checked.
So assume that $G$ has order $5$.

If $G$ contains no $4$- or $5$-cycle,
then we may apply
Corollary~\ref{P:cyc3}.
If $G$ has a $5$-cycle,
then arrange the $5$ vertices on a circle in the plane
so that each edge in the $5$-cycle joins vertices
that are not consecutive on the circle
(so the $5$-cycle is a pentagram).
No matter what other edges lie in $G$,
an outside obstacle can block all non-edges
(see Figure~\ref{F:5cycle} for an illustration).

\begin{figure}[htbp]
\begin{center}
\includegraphics[]{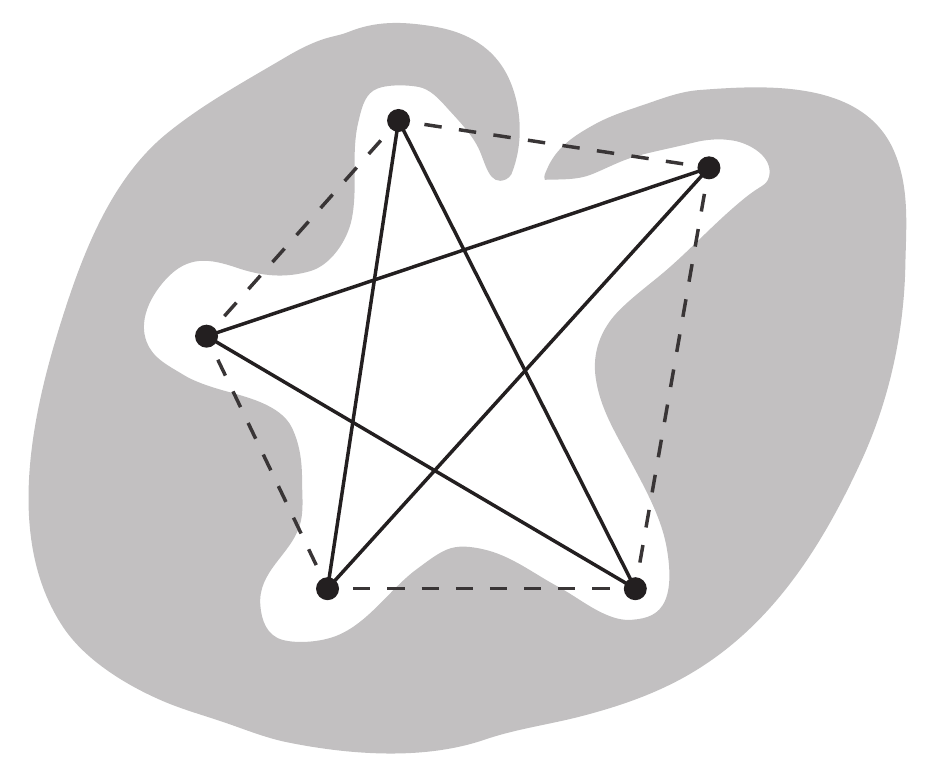}
\caption{An embedding of a $5$-cycle permitting the blocking of any remaining possible edges
(drawn as dashed lines) by a single outside obstacle.}
\label{F:5cycle}
\end{center}
\end{figure}

So we may assume that $G$ contains no $5$-cycle,
but does contain a $4$-cycle.
Choose a $4$-cycle in $G$, and let
$x$ be the vertex of $G$ that does not lie on this $4$-cycle.
Since $G$ has minimum degree at least $2$,
$x$ must be adjacent to at least $2$ vertices of the $4$-cycle.
The neighbors of $x$ cannot include two consecutive vertices
on the $4$-cycle,
since then $G$ would have a $5$-cycle.
Consequently, $x$ has degree exactly $2$,
and its neighbors are nonconsecutive vertices on the $4$-cycle.
It remains to consider adjacencies among vertices in the $4$-cycle.
The non-neighbors of $x$ cannot be adjacent, as this
would form a $5$-cycle in $G$.

The only question left
is whether the neighbors of $x$ are adjacent.
Thus there are exactly $2$ graphs that satisfy our assumptions:
$K_{2,3}$ and $K_{1,1,3}$.
Both are easily shown to have outside obstacle number $1$.\end{proof}

\begin{corollary}
Let $k$ be the minimum order of a graph $G$
with $\obs(G) = 2$.
Then $6\le k\le 10$.
Similarly, the minimum order $k$ of a graph $G$ with
$\obsout(G) = 2$ satisfies $6\le k\le 10$.
\end{corollary}

\begin{proof}
The lower bound follows from Proposition~\ref{P:ord5}.
For the ordinary obstacle number,
the upper bound follows from the result of
Pach \& Sari\"{o}z (Theorem~\ref{T:kstar55}).
The obstacle representation constructed by
Pach \& Sari\"{o}z~\cite[Fig.~1]{PacSar11}
uses an outside obstacle,
and so $\obsout(K^*_{5,5}) = 2$ as well.
The upper bound for the outside obstacle number follows.\end{proof}

We now consider what properties a small graph with $\obs(G) > 1$
must have.
We show that such a graph cannot be a subgraph of $K_{4,4}$.

\begin{proposition}\label{P:subgK44}
Let $G$ be a subgraph of $K_{4,4}$.
Then $\obsout(G)\le 1$.
\end{proposition}

\begin{proof}
Let $a$, $b$ be the smallest values such that $G$
is a subgraph of $K_{a,b}$.
Let $H$ be the subgraph of $K_{a,b}$ formed by those
edges that do \emph{not} lie in $G$.

If $G$ has a vertex of degree $0$ or $1$,
then we may apply induction, by Proposition~\ref{P:pendant}.
So assume $G$ has minimum degree at least $2$.
Then $H$ has maximum degree at most $2$.

We observe that,
if there is a set $S$ of vertices containing at most $2$
vertices from each partite set,
such that $S$ contains an endpoint of each edge in $H$,
then we can draw $G$ so that every non-edge is blocked by an outside obstacle,
by placing the two partite sets in (near) lines,
with the vertices in $S$ at the ends of their respective lines.
(See Figure~\ref{F:setSinK44} for an illustration.)
We will use this observation repeatedly in the remainder of this proof.

\begin{figure}[htbp]
\begin{center}
\includegraphics[]{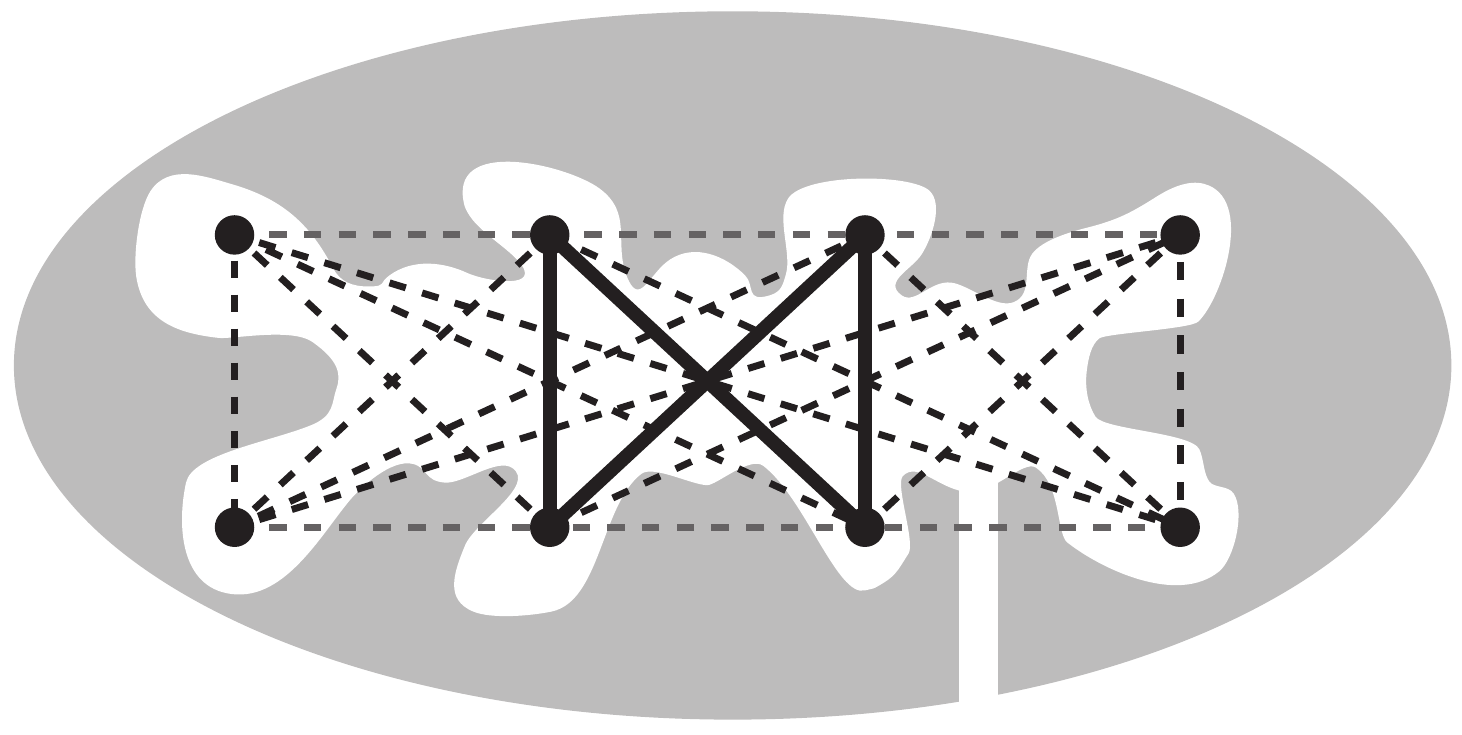}
\caption{Illustration of part of
the proof of Proposition~\ref{P:subgK44}.
The two partite sets of $G$ are formed by the upper and lower four vertices,
respectively.
The set $S$ consists of the two rightmost and the two leftmost
vertices.
Light dashed lines are non-edges within each partite set.
Dark dashed lines are the non-edges between the partite sets,
that is, the edges of $H$;
each has at least one endpoint in $S$.
The drawing illustrates how the gray outside obstacle
can block every non-edge of $G$.}
\label{F:setSinK44}
\end{center}
\end{figure}

If $H$ does not contain a matching of size $4$,
then $H$ has a vertex cover $C$ containing at most $3$ vertices,
by the K\"{o}nig-Egerv\'{a}ry Theorem
(see, e.g., Bondy \& Murty~\cite[p.~74]{BonMur76}).
If $C$ includes members of both partite sets,
then let $S = C$.
Otherwise, replace one of the vertices in $C$
with the other endpoints of the (at most $2$) edges of $H$
it is incident with;
again, we have our set $S$.

So we may assume that $H$ has a matching of size $4$.
Thus $a = b = 4$.
Since $G$ has minimum degree at least $2$,
we have two cases:
in the first case $G$ contains two vertex-disjoint $4$-cycles;
in the second $G$ contains an $8$-cycle.
In the former case, the vertices of one of the $4$-cycles
form our set $S$.
In the latter case, if $G$ contains at least one edge $e = xy$
that is not in the $8$-cycle, then $e$ and $3$ edges of the
$8$-cycle form a $4$-cycle in $G$.
Let $S$ consist of the $4$ vertices of $G$ that
do not lie on this $4$-cycle.

It remains only to handle the case when $G$ is an $8$-cycle
with no additional edges.
In this case there is no set $S$ with the properties we are looking for.
However, we can draw $G$ with every edge blocked by an outside obstacle as in Figure~\ref{F:8cycle}.
\end{proof}

\begin{figure}[htbp]
\begin{center}
\includegraphics[width=.3\linewidth]{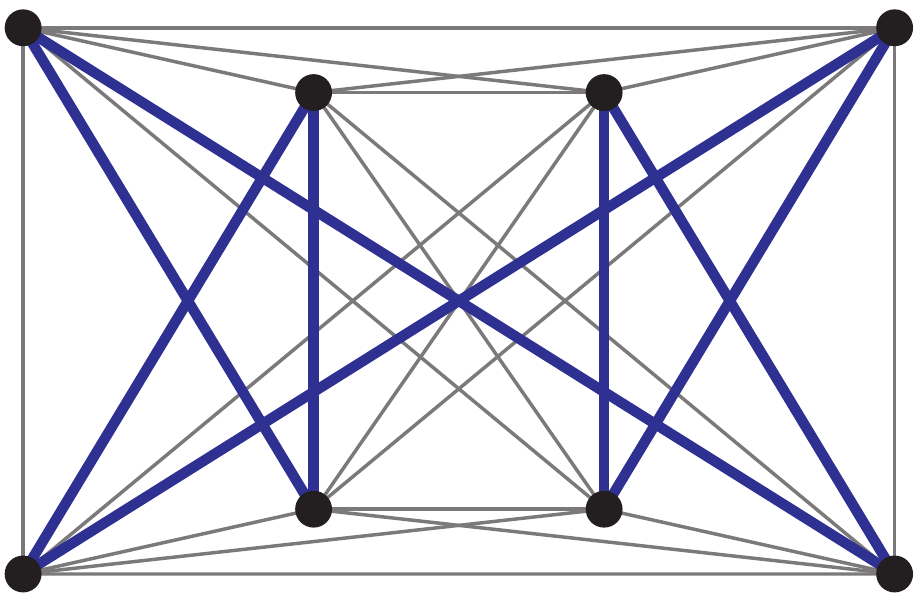}
\caption{An obstacle representation of an $8$-cycle
using a single outside obstacle.
Thin lines represent non-edges.}
\label{F:8cycle}
\end{center}
\end{figure}

\section{Outside Obstacle Number} \label{S:obsout}

We now begin our development of tools for explicitly determining
the (outside) obstacle number of a particular graph.
We will apply these tools to various planar graphs.
In this section, we cover results related to outside obstacle number;
we will consider the ordinary obstacle number in the following sections.

Our ideas are based on the Satisfiability Problem (SAT).
An instance of this problem is a particular kind of Boolean formula.
In this context, a \emph{variable} is a Boolean variable;
i.e., its value is either \emph{true} or \emph{false}.
A \emph{literal} is either a variable or the negation of a variable.
A \emph{clause} is the inclusive-OR of one or more literals.
For example, given variables $x_1$, $x_2$, $x_3$ and $x_4$,
the following is an example of a clause:
\[
x_1 \vee \neg x_2 \vee x_4
\]
An instance of SAT consists of a number of clauses.
The instance is said to be \emph{satisfiable}
if there exists a truth assignment for the variables
such that every clause in the instance is true;
i.e., a truth assignment in which each clause contains at least
one true literal.

We will show how to construct, for each graph $G$,
a SAT instance
encoding
necessary conditions for the existence of an obstacle representation
of $G$ using no interior obstacles.
Thus, if we can show that the instance is not satisfiable,
then we know that $\obsout(G) \ge 2$.
There are a number of freely available,
high-quality implementations of algorithms
to determine satisfiability of a SAT instance.
Using these, we will construct computer-aided
proofs that $\obsout \ge 2$
for various planar graphs.

When interpreting the effect of the satisfiability or non-satisfiability
of the  SAT instances on the obstacle number problem,
it is important to note that \emph{if} the SAT instance is
\emph{not} satisfiable,
then we are guaranteed that it is impossible to draw the graph
using a single outside obstacle;
i.e., $\obsout \geq 2$.
However, if the SAT instance \emph{is} satisfiable,
then it does not follow that $\obsout \leq 1$;
satisfiability is a necessary but not sufficient condition
for concluding that $\obsout \leq 1$.


For $a,b$ and $c\in\RR^2$, we say that $abc$ is a \emph{clockwise triple}
if $a$, $b$ and $c$ appear in clockwise order.
More formally, identifying each point in $\RR^2$
with the corresponding point in the plane $z=1$ in $\RR^3$ via the map $a = (a_{1},a_{2}) \mapsto (a_{1},a_{2},1)$,
we denote
the determinant of the $3\times 3$ matrix with columns $a$, $b$ and $c$
by $[abc]$.
Then $abc$ is a clockwise triple if $[abc] < 0$.
We similarly define \emph{counter-clockwise triple}.
Note that for a triple $abc$, exactly one of the following is true:
$abc$ is clockwise,
$abc$ is counter-clockwise, or
$a$, $b$ and $c$ are collinear (in which case $[abc]=0$).

The following two lemmas give properties that hold for all
point arrangements in the plane.
We will use these properties to construct clauses for our SAT instances.

\begin{lemma}[$4$-Point Rule] \label{L:4pt}
Let $a$, $b$, $c$, $d$ be distinct points in $\RR^2$.
If $abc$, $acd$, and $adb$ are clockwise triples, then $bcd$ must also be a clockwise triple.
\end{lemma}

\begin{proof}
The $4$-Point Rule is equivalent to what D.~Knuth
called the \emph{interiority}
property of triples of points;
see Knuth~\cite[p.~4, Axiom~4]{Knu92}. In addition, the $4$-Point Rule is a consequence of the fact that a point configuration is a rank 3 \emph{acyclic} oriented matroid. 
 Intuitively, the $4$-Point Rule follows from the observation that if $abc$ is a clockwise-oriented triple, then $d$ cannot be ``outside'' all the lines $ab$, $bc$, $cd$ (where ``outside the lines'' means on the negative side of the closed halfplanes determined by the orientation of $abc$). See Figure \ref{fig:4ptRule}.
\end{proof}

\begin{figure}[htbp]
\begin{center}
\begin{tikzpicture}[scale=.5, dot/.style={circle,inner sep=2pt,fill,label={$#1$},name=#1},
  extended line/.style={shorten >=-#1,shorten <=-#1},
  extended line/.default=1cm]
\node[dot=b]  at (-1,0){};
\node[dot=a]at (1,2){};
\node[dot=c] at (-3,3){};
\draw [extended line=0.5cm] (a) -- (b);
\draw [extended line=0.5cm] (b) -- (c);
\draw [extended line=0.5cm] (c) -- (a);

\draw[->] ($(a)!.5!(b)$) -- ($(a)!.5!(b)!.3!90:(b)$);
\draw[->] ($(b)!.5!(c)$) -- ($(b)!.5!(c)!.3!90:(c)$);
\draw[->] ($(c)!.5!(a)$) -- ($(c)!.5!(a)!.3!90:(a)$);



\end{tikzpicture}
\caption{If $abc$ is a clockwise-oriented triangle, then $d$ cannot be on the negative side with respect to the orientation of $abc$ (negative side indicated by the arrows) of each of the halfplanes determined by the sides of $abc$. Therefore, if $acd$ and $adb$ are also clockwise-oriented, then $bcd$ must be clockwise too.}
\label{fig:4ptRule}
\end{center}
\end{figure}
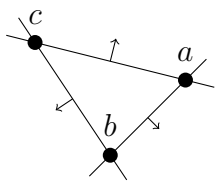

We encode the $4$-Point Rule as follows. 
For each triple $abc$ we introduce a variable $x_{abc}$ representing the statement that $abc$ is a clockwise triple. By Observation~\ref{O:perturb} we may assume that no three vertices
are collinear, so $\lnot x_{abc}$  represents the statement that $abc$ is a counterclockwise triple. The $4$-point rule is represented by the clause
\begin{equation}\label{Cl:4pt}  \lnot x_{abc} \lor \lnot x_{acd} \lor \lnot x_{adb} \lor  x_{bcd}.\end{equation}

\begin{lemma}[$5$-Point Rule] \label{L:5pt}
Let $a$, $b$, $c$, $d$ and $e$ be distinct points in $\RR^2$.
If $abc$, $acd$, $ade$, and $abe$ are clockwise triples,
then either
\be[(i)]
\item both $abd$ and $ace$ are clockwise triples, or
\item both $abd$ and $ace$ are counter-clockwise triples.
\ee
\end{lemma}

\begin{proof}
This is equivalent to what D.~Knuth
called the \emph{transitivity}
property of triples of points;
see Knuth~\cite[p.~4, Axiom~5]{Knu92}.
It also
follows from the Grassman-Pl\"{u}cker relation
\[
[abc][ade]-[abd][ace]+[abe][acd] = 0;
\]
see Bj\"{o}rner \textit{et al.}~\cite[pp.~6--7]{BjoLasStu99}.
\end{proof}


The $5$-Point Rule
is  represented by the following two clauses:
\begin{gather}
\label{Cl:5pt1}
\neg x_{abc} \vee \neg x_{acd} \vee \neg x_{ade} \vee \neg x_{abe}
  \vee \phantom{\neg} x_{abd} \vee \neg x_{ace};\\
\label{Cl:5pt2}
\neg x_{abc} \vee \neg x_{acd} \vee \neg x_{ade} \vee \neg x_{abe}
  \vee \neg x_{abd} \vee \phantom{\neg} x_{ace}.
\end{gather}


In constructing our SAT instance, we will make use of
the following lemma,
which describes a restriction
that must hold for an obstacle representation
without interior obstacles.

\begin{lemma} \label{L:extobs}
Suppose we are given an obstacle representation of a graph $G$
using no interior obstacles.
Let $ab$ be a non-edge of $G$.
Then there exists a half-plane $H$ determined by
the line $\overleftrightarrow{ab}$
such that, for each $a,b$-path $P$ in $G$,
some internal vertex of $P$ lies in $H$.
\end{lemma}

\begin{proof}
Denote by $H^{+}$ and $H^{-}$ the two closed half-planes
determined by the line $\overleftrightarrow{ab}$.

Suppose that there exist two distinct $a,b$-paths $P_{1}$ and $P_{2}$
in $G$
such that (without loss of generality)
$P_{1}\subseteq H^{+}$ and $P_{2}\subseteq H^{-}$.
Perturbing slightly if necessary,
we assume that no internal vertex of $P_{1}$ or $P_{2}$
lies on the line $\overleftrightarrow{ab}$
(see Observation~\ref{O:perturb}).
We observe then that $P_{1}\cup P_{2}$ forms a closed curve in the plane;
the open line segment from $a$ to $b$ lies in a bounded component
of the complement of this closed curve.
By assumption there is no obstacle lying in such a bounded component,
and thus there is nothing to block the line segment $ab$,
a contradiction.
\end{proof}

Given a graph $G$,
we use Lemmas~\ref{L:4pt}, \ref{L:5pt}, and \ref{L:extobs}
to create a SAT instance
whose satisfiability is a necessary condition for the
existence of a obstacle representation of $G$
using no interior obstacles.
If this SAT instance is not satisfiable,
then we may conclude that graph $G$ requires an interior obstacle,
that is, that $\obsout(G) > 1$.

Since we are placing the vertices of $G$ in the plane,
the $4$-Point Rule (Lemma~\ref{L:4pt})
must hold for any $4$ vertices in the graph,
the $5$-Point Rule (Lemma~\ref{L:5pt})
must hold for any $5$ vertices in the graph,
and by Observation~\ref{O:perturb}
we may assume that no three vertices are collinear.
%
Our SAT instance includes clauses from the $4$-Point Rule corresponding to clause \eqref{Cl:4pt}
for every set of $4$ vertices of our graph,
and every permutation of these $4$ vertices.
It also includes clauses from the $5$-Point Rule corresponding to clauses \eqref{Cl:5pt1} and \eqref{Cl:5pt2}
for every set of $5$ vertices of our graph,
and every permutation of these $5$ vertices.

Note that there are six ways to say vertices $a, b$ and $c$
lie in clockwise order.
Variables corresponding to even permutations of $abc$
($x_{abc}$, $x_{bca}$, $x_{cab}$)
represent equivalent statements.
Variables corresponding to
odd permutations of $abc$
($x_{bac}$, $x_{acb}$, $x_{cba}$)
represent their negations.

When we construct our SAT instance,
we may choose one of the variables from among $\{x_{abc}, x_{bca}, x_{cab} x_{bac}, x_{acb}, x_{cba}\}$ as the canonical variable;
we represent the other five using either the canonical
variable or its negation, as appropriate.

Additionally, the actions of even permutations of $\{b,c,d\}$ on clauses~(\ref{Cl:4pt} and~(\ref{Cl:5pt1})
result in statements equivalent to the original; 
likewise an even permutation of $\{c,d,e\}$
does not alter clause~(\ref{Cl:5pt2}).
This reduces the number of clauses required by a factor of $3$.
Thus, for an $n$-vertex graph,
our SAT instance includes
$\binom{n}{4}\cdot 4! / 3 = 8\binom{n}{4}$
clauses based on the $4$-Point Rule
and
$2\binom{n}{5}\cdot 5! / 3 = 80\binom{n}{5}$
clauses based on the $5$-Point Rule.

We similarly construct SAT clauses based on Lemma~\ref{L:extobs}.
Since this lemma concerns a line defined by two
points, and which side of this line certain other points lie on,
the lemma can be stated
in terms of clockwise or counter-clockwise triples.
Specifically,
this lemma implies that,
for each pair of nonadjacent vertices $a$, $b$ of $G$,
one of the two half-planes
determined by segment $ab$
is ``special'';
that is, this half-plane
contains at least one internal vertex from each $a,b$-path in $G$.
We create a new variable $s_{ab}$ representing the statement
that the special half-plane is the half-plane containing points $p$
such that $abp$ is a clockwise triple.

Let $a,s,t,\dots,u,b$ be the sequence of vertices in some $a,b$-path $P$.
Then the following clauses represent
the statement of Lemma~\ref{L:extobs} for $P$:
\begin{gather}
\label{Cl:extobs1}
\neg s_{ab} \vee \phantom{\neg} x_{abs} \vee \phantom{\neg} x_{abt} \vee
  \dots \vee \phantom{\neg} x_{abu}\\
\label{Cl:extobs2}
\phantom{\neg} s_{ab} \vee \neg x_{abs} \vee \neg x_{abt} \vee
  \dots \vee \neg x_{abu}
\end{gather}
As with the $x$ variables,
we choose one of $s_{ab}$ and $s_{ba}$
to be the canonical variable,
and we represent the other by its negation.

\begin{observation} \label{O:sat-obsout}
Let $G$ be a graph.
If $\obsout(G) \le 1$,
then the SAT instance consisting of all clauses
of the forms (\ref{Cl:5pt1})--(\ref{Cl:extobs2})---using
canonical variables, as discussed above---is
satisfiable. \end{observation}

Note that if all that is known is that the SAT instance is satisfiable,
then we have no definitive information about $\obsout(G)$,
although the satisfiability of the instance may
provide
us with a starting point for finding---by hand---a useful embedding.

Using these ideas,
we can determine the exact value
of the outside obstacle number for some special planar graphs.

Following \cite{Joh66}, for $n\ge 3$
we define the \emph{gyroelongated $n$-bipyramid}
to be a convex polyhedron formed by adding pyramids
to the top and bottom base of the $n$-antiprism;
see Figure~\ref{F:gyroEx}.
The gyroelongated square bipyramid,
when constructed using equilateral triangles,
is also known as the Johnson solid $J_{17}$.
The gyroelongated pentagonal bipyramid,
again when constructed of equilateral triangles,
is the regular icosahedron.
We denote the skeleton of the gyroelongated $n$-bipyramid by $X_{n}$.
We also refer to $X_{5}$ as $I$,
since it is the icosahedron.

The graph $X_n$
can be constructed as two disjoint $n$-wheels,
connected by a $2n$-cycle that alternates
between vertices of the wheel boundaries taken cyclically.
Figures~\ref{F:gyroEx}, 
\ref{F:IcosPix}, 
\ref{F:gyroDipyramids}, 
and
\ref{F:Gen-3-Icos} 
show gyroelongated $n$-bipyramids for various values of $n$,
while
Figure~\ref{F:gyroGeneral} 
illustrates the general case.
In each of these figures,
the wheel boundaries are are labeled with consecutive numbers and are shown in green,
the spokes of the wheel in blue,
and the connecting cycle in black; the cycle connecting the two wheel boundaries correspond to the sequence of labeled vertices $1,\blue{1},2,\blue{2},\ldots,n,\blue{n}$.
Non-edges are shown with thin gray or pink lines.

\begin{figure}[htbp]
\begin{center}
\ffigbox{
\begin{subfloatrow*}[3]
\subfloat[$X_{3}$]{  \includegraphics[width=.2\linewidth]{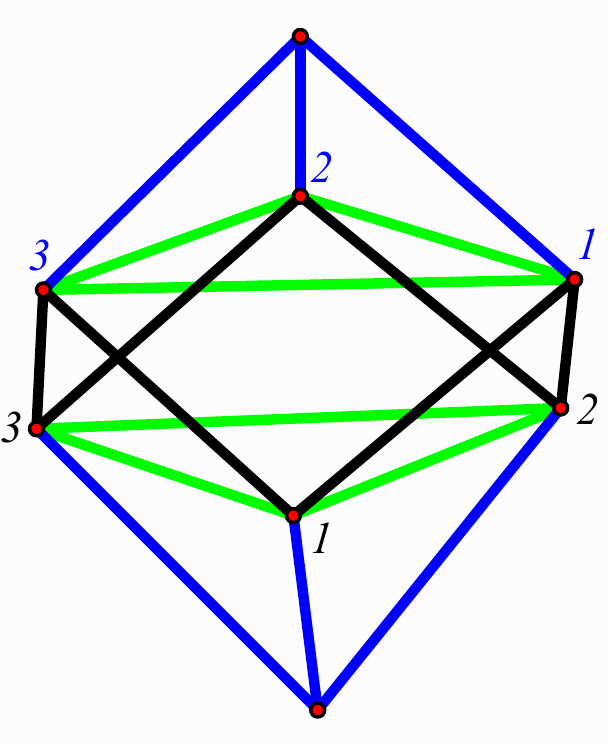}
 }
\subfloat[$X_{4}$]{  \includegraphics[width=.2\linewidth]{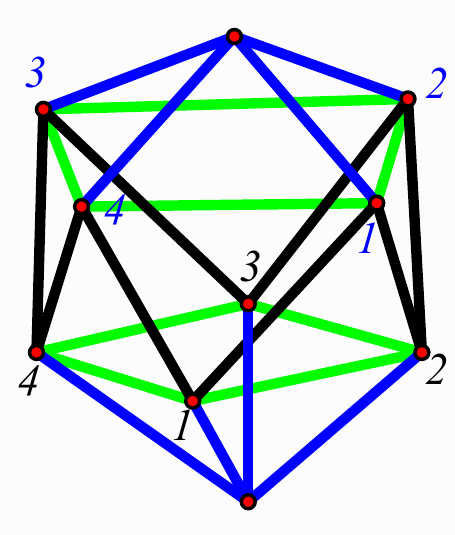}
 }
\subfloat[$X_{5} = I$]{ \includegraphics[width=.2\linewidth]{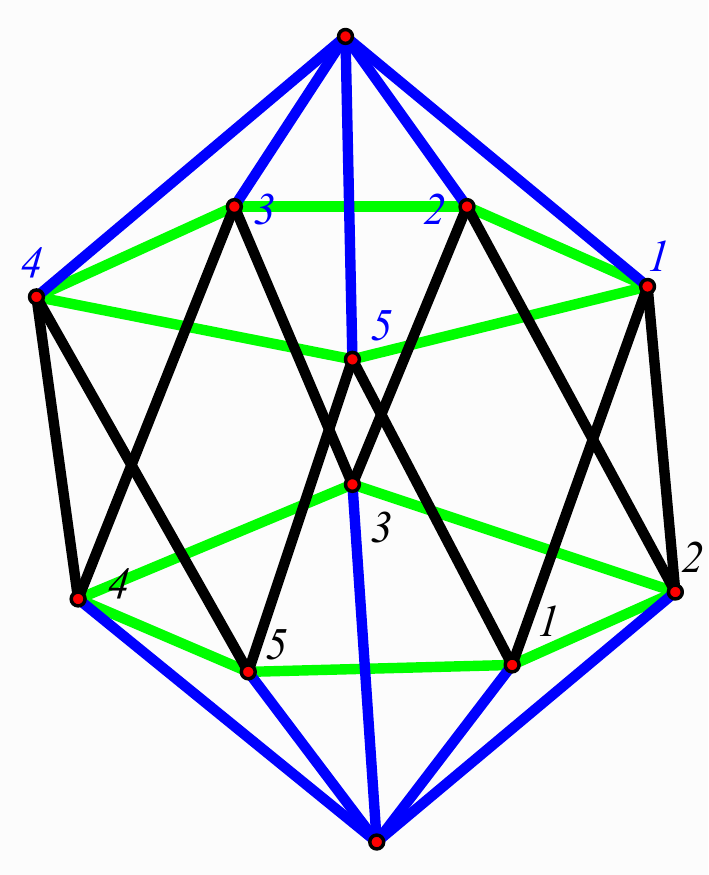}
  }
\end{subfloatrow*}
}{
\caption{Gyroelongated $n$-bipyramid skeleta; the antiprisms are highlighted in black and green, while the pyramids erected on the bases are shown in blue.}
\label{F:gyroEx}}
\end{center}
\end{figure}

\begin{proposition} \label{P:gesdp-obsout}
All of the following hold.
\be
\item $\obsout(X_{4}) = 2$.
\item $\obsout(I)= \obsout(X_{5}) = 2$.
\item $\obsout(X_{6}) = 2$.
\ee
\end{proposition}

\begin{proof}
The lower bounds were found using a computer.
Two of the authors independently
created SAT instances as described above,
using clauses derived from Lemmas~\ref{L:4pt} (the $4$-Point Rule),
\ref{L:5pt} (the $5$-Point Rule),
and \ref{L:extobs},
as discussed above.
See~\cite{Cha17} for software to generate the SAT instances.
For each graph,
standard SAT solvers
(we used MiniSat \cite{MiniSat}, PicoSAT \cite{PicoSAT},
and zChaff \cite{zChaff})
indicate that
the SAT instance is not satisfiable.
Thus, by Observation~\ref{O:sat-obsout},
we have $\obsout \ge 2$ for each graph.

For the upper bounds,
we exhibit an obstacle representation of each graph
using two obstacles, one of which is an outside obstacle.
Figure~\ref{F:IcosPix} shows embeddings of the icosahedron,
while Figure~\ref{F:gyroDipyramids} shows embeddings of $X_{4}$ and $X_{6}$.
\end{proof}

Note that Figure~\ref{Fp:IcosSym} exhibits $3$-fold dihedral symmetry,
while the other representations have only one vertical mirror of symmetry.

\begin{figure}[htbp]
\begin{center}
\ffigbox{
\begin{subfloatrow*}
\subfloat[]{\label{Fp:IcosSym}\includegraphics[width=.6\linewidth]{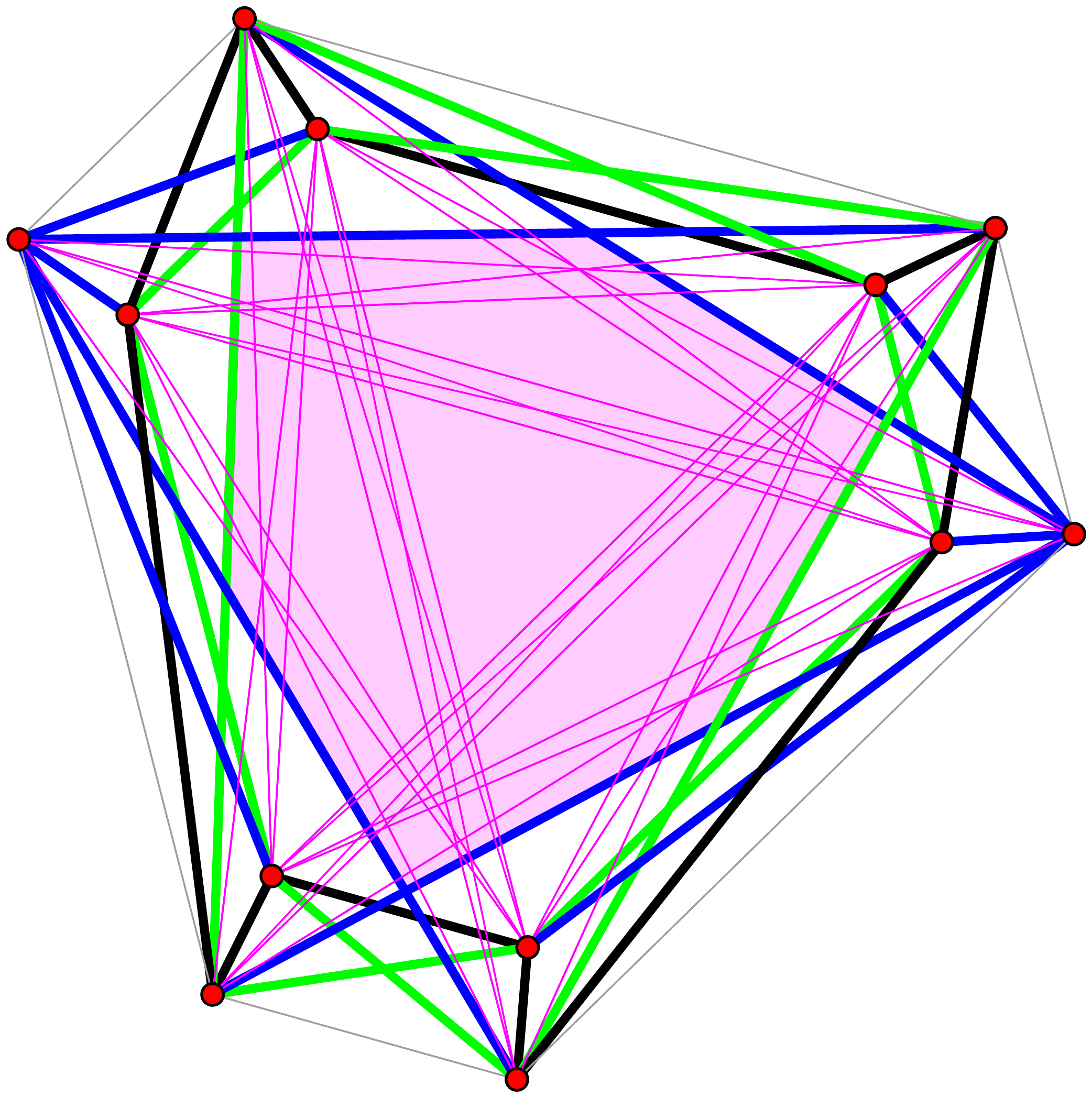}}
\end{subfloatrow*}
\begin{subfloatrow*}
\subfloat[]{\label{Fp:IcosGyro}\includegraphics[width=.8\linewidth]{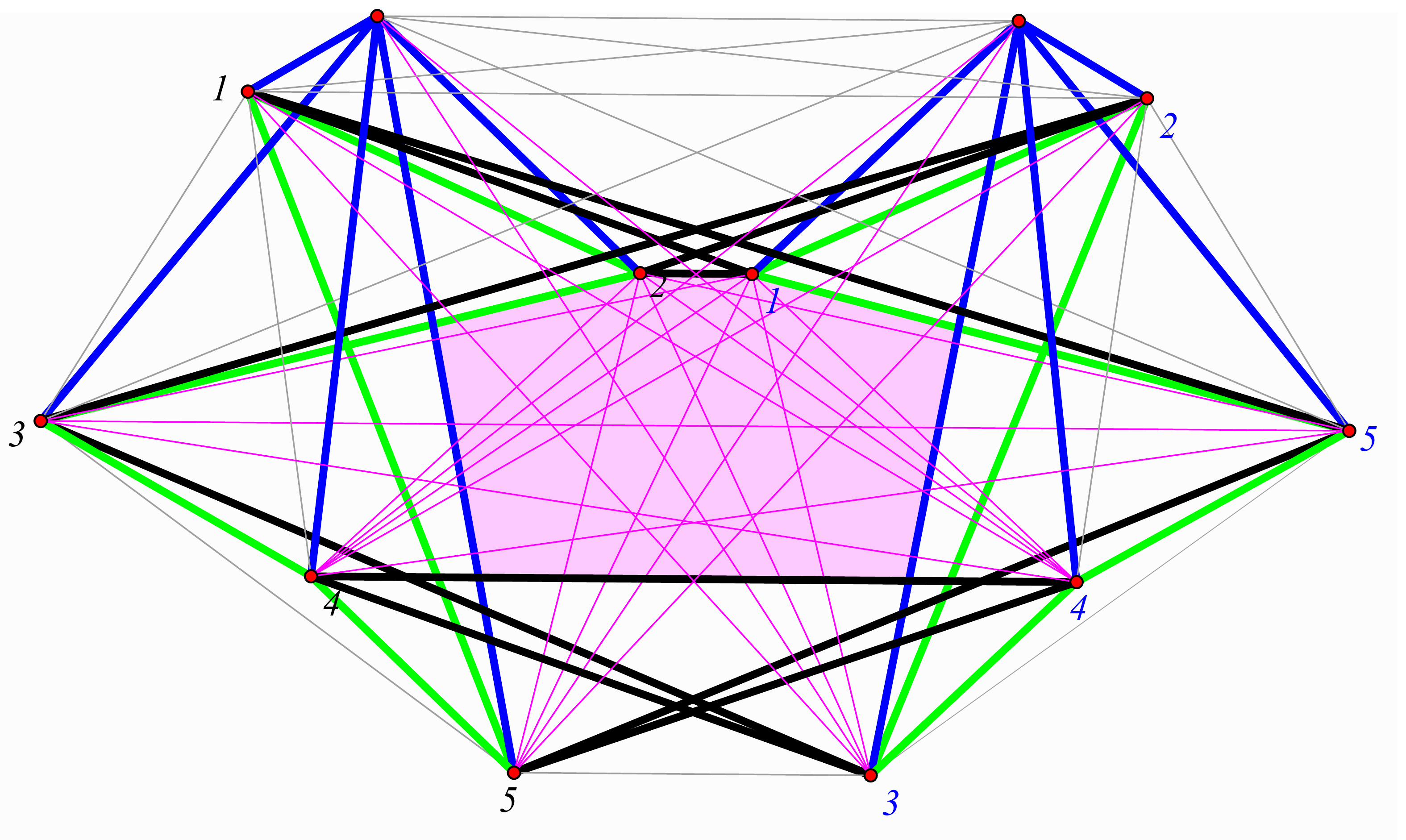}}
\end{subfloatrow*}}{
\caption{Two $2$-obstacle embeddings of the icosahedron.
The interior obstacle is highlighted in pale magenta,
and non-edges blocked by that obstacle are shown with thin magenta lines.
The other obstacle is the outside obstacle,
and non-edges blocked by that obstacle are shown with thin gray lines.}
\label{F:IcosPix}}
\end{center}
\end{figure}

\begin{figure}[htbp]
\begin{center}
\ffigbox{
\begin{subfloatrow*}
\subfloat[$X_{4}$]{\label{Fp:gyroQuad}\includegraphics[width=.8\linewidth]{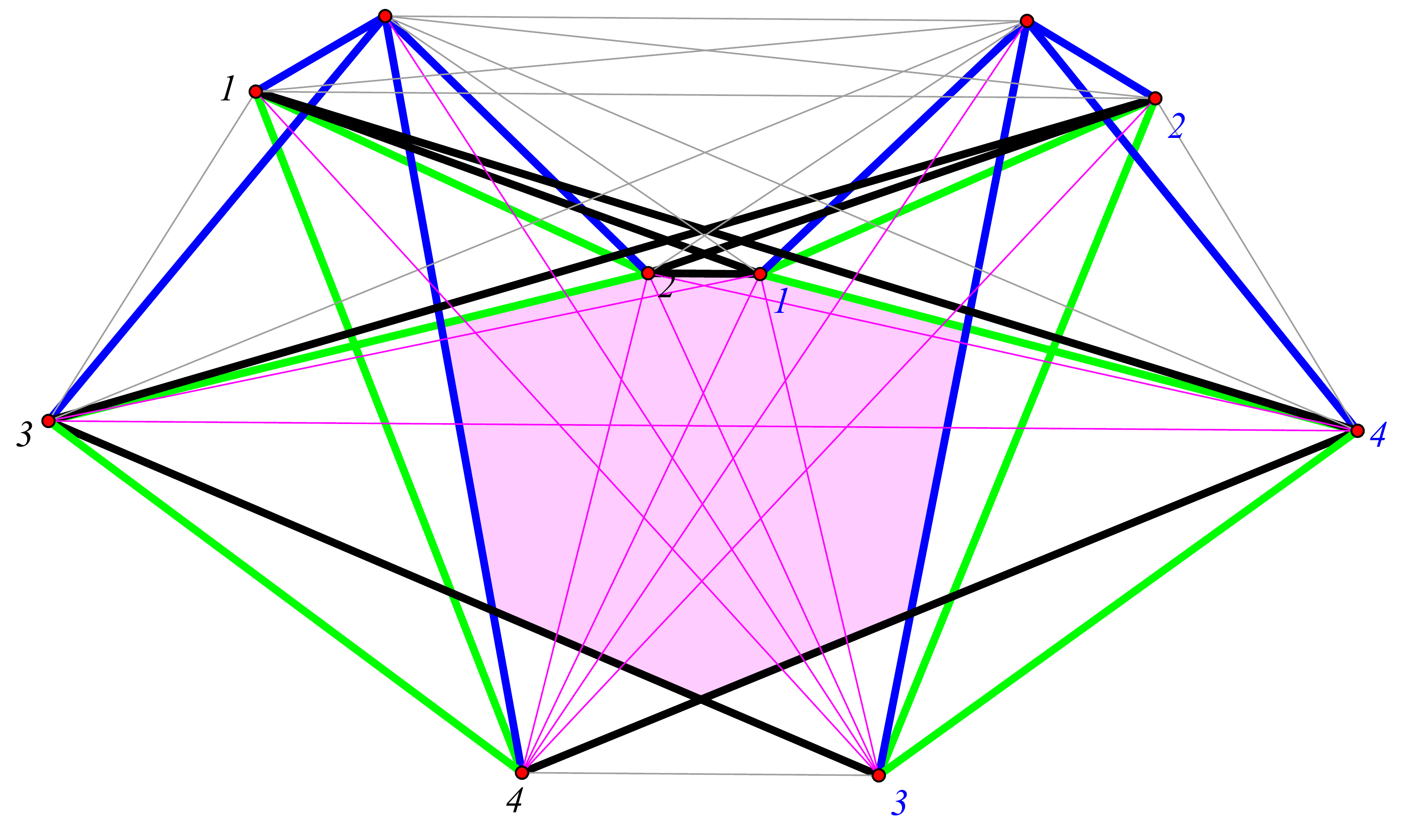}}
\end{subfloatrow*}
\begin{subfloatrow*}
\subfloat[$X_{6}$]{\label{Fp:gyroHex}\includegraphics[width=.8\linewidth]{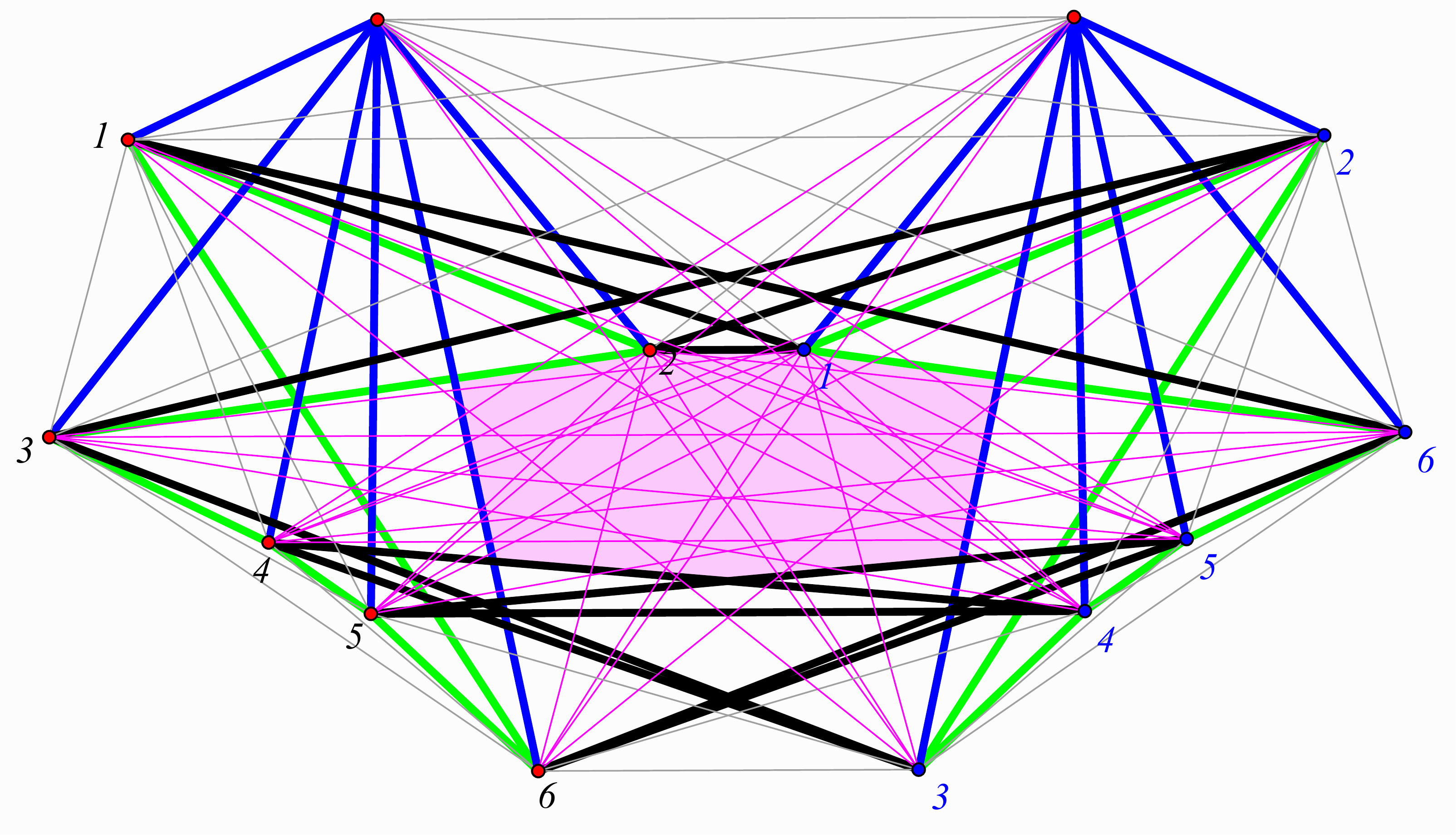}}
\end{subfloatrow*}}{
\caption{Two-obstacle embeddings of gyroelongated $n$-bipyramids; $n = 4,6$.}
\label{F:gyroDipyramids}}
\end{center}
\end{figure}

Figure~\ref{F:gyroGeneral} shows how the embeddings
in Figures~\ref{Fp:IcosGyro} and \ref{F:gyroDipyramids}
can be generalized to a $2$-obstacle embedding of $X_n$
for arbitrarily large $n$.
Thus we observe the following.

\begin{theorem}
Let $n\ge 3$.
Then $\obsout(X_{n}) \leq 2$.
\end{theorem}

\begin{figure}[htbp]
\begin{center}
\includegraphics[width=\linewidth]{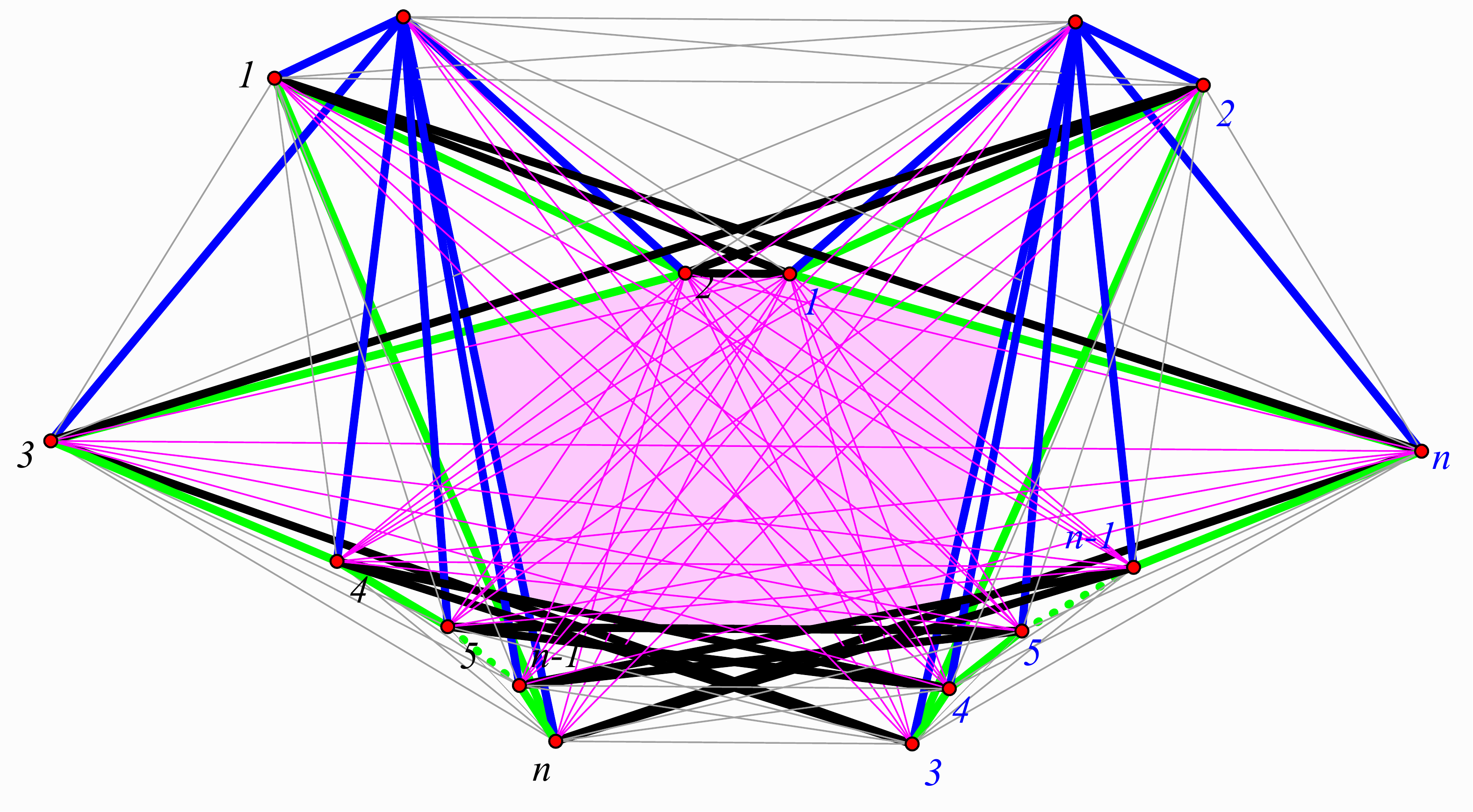}
\caption{A two-obstacle embedding of a gyroelogated $n$-bipyramid.}
\label{F:gyroGeneral}
\end{center}
\end{figure}

For the gyroelongated $3$-bipyramid
we have $\obsout(X_{3}) = 1$
(see Figure~\ref{F:Gen-3-Icos}),
while for $n = 4,5,6$ we have $\obsout(X_n) = 2$
by Proposition~\ref{P:gesdp-obsout}.
It seems likely that, for all larger values of $n$,
the graph $X_n$ cannot be represented without an interior obstacle;
however the corresponding SAT instances are too large
for our computational methods to be feasible.

\begin{conjecture} If $n \geq 4$, then $\obsout(X_{n}) = 2$.\end{conjecture}

\begin{figure}[htbp]
\begin{center}
\ffigbox{
\begin{subfloatrow*}
\subfloat[]{\includegraphics[width=.4\linewidth]{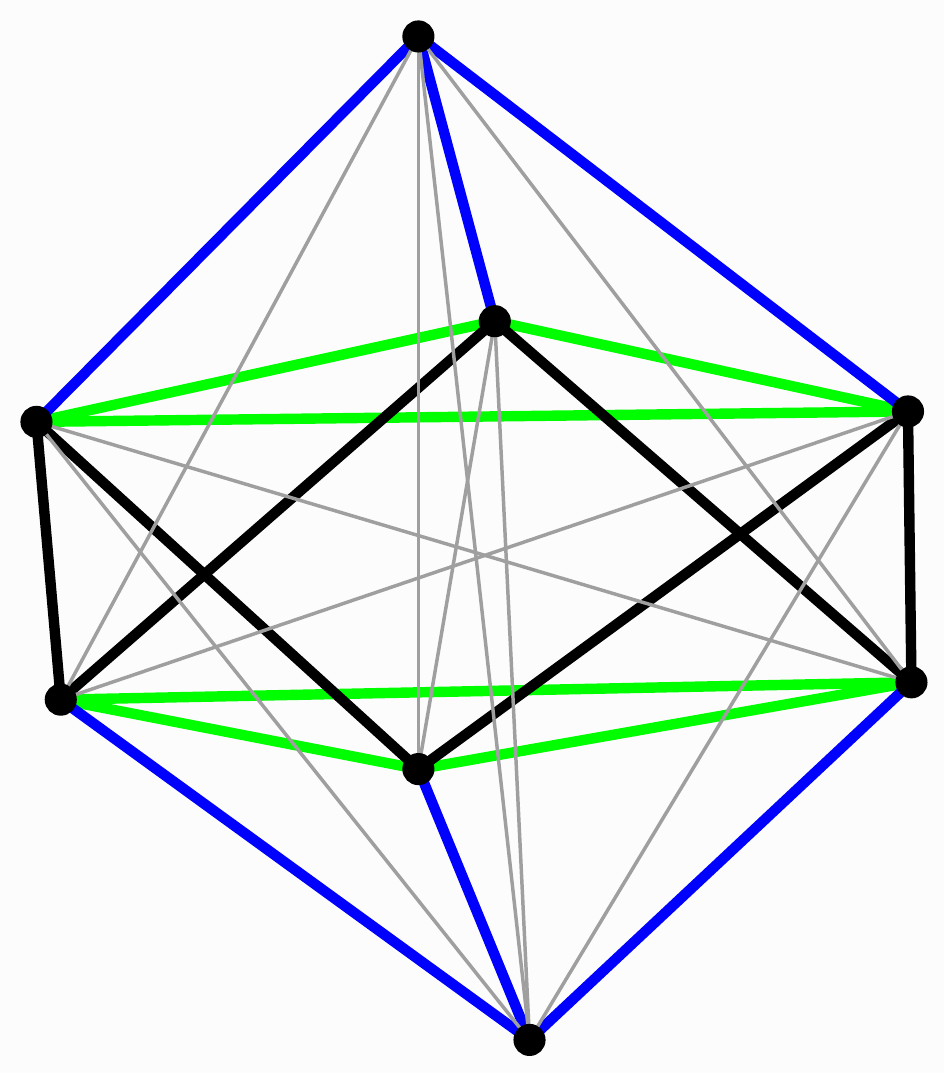}}
\subfloat[]{\includegraphics[width=.4\linewidth]{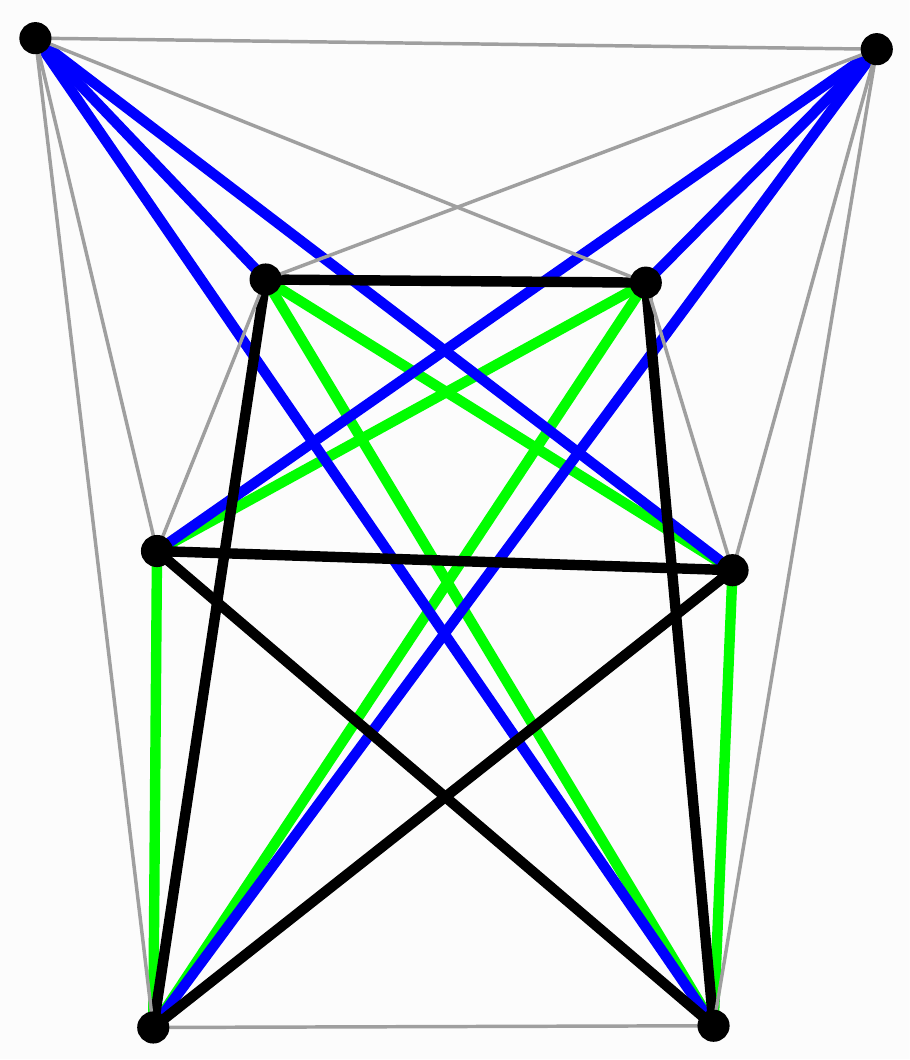}}
\end{subfloatrow*}}{
\caption{(A) The gyroelongated $3$-bipyramid $X_{3}$ and (B) a one-obstacle embedding.}
\label{F:Gen-3-Icos}}
\end{center}
\end{figure}

\section{Obstacle Number, Part I} \label{S:obs1}

We have shown that there exist planar graphs with
outside obstacle number greater than $1$.
We now turn our attention to the ordinary obstacle number.

Knowing that certain graphs require an interior obstacle
allows us to reason about the obstacle number of a larger graph;
in particular this will allow us to argue that certain planar graphs
have obstacle number greater than or equal to $2$.

We denote the disjoint union of graphs $G_1$, $G_2$
by $G_1\cupdotop G_2$.

\begin{lemma} \label{L:disjun}
Let $G_1$ and $G_2$ be graphs
such that
$\obsout(G_1) > 1$ and $\obsout(G_2) > 1$.
Then $\obs(G_1\cupdotop G_2) > 1$.\end{lemma}

\begin{proof}
Let $G_1$, $G_2$ be as in the statement of the result,
and
let $G=G_{1}\cupdotop G_{2}$.
Suppose for a contradiction that
we have an obstacle representation of $G$ using only
one obstacle $O$.

Based on this representation of $G$,
construct an
obstacle representation of $G_1$ by eliminating
those vertices and edges that do not lie in $G_1$
(see Observation~\ref{O:induced}).
Let $O_1$ be the largest connected open set
that contains $O$ but does not meet the drawing of $G_1$.
We may replace $O$ by $O_1$ in our obstacle representation of $G_1$.
Since $\obsout(G_1) > 1$,
this obstacle $O_1$ must be an interior obstacle.
We similarly construct an obstacle representation
of $G_2$ using a single interior obstacle $O_2$.
Note that each of $O_1$, $O_2$ is the interior of a bounded
polygonal region,
and that $O\subseteq O_1\cap O_2$.

Denote the topological closure of a subset $X$ of the plane
by $\overline{X}$.
Let $\bv{v}$ be a nonzero $2$-vector,
and consider the family $\mathcal{H}$ of parallel closed half-planes
determined by $\bv{v}$.
As we sweep this family across $G$,
beginning outside of $G$,
there will be some $H\in\mathcal{H}$
that is the first member of $\mathcal{H}$
with the property that
$H$ meets both $\overline{O_1}$ and $\overline{O_2}$.

For at least one of $O_1$, $O_2$,
half-plane $H$ meets only the boundary of the
obstacle.
Say this holds for $O_1$;
then $H\cap O_1 = \varnothing$,
and so $H\cap O = \varnothing$.

The set $\overline{O_1}$
is bounded by line segments contained in the edges of $G_1$.
Thus some vertex $p$ of $G_1$ lies in $H$.
Similarly, some vertex $q$ of $G_2$ lies in $H$.
Since $pq$ is not an edge of $G$,
the line segment $pq$
must be blocked by obstacle $O$.
However, this line segment
lies entirely in $H$, and so it does not meet $O$,
a contradiction.
\end{proof}

As a consequence, we see that
each of the graphs
$X_{4}\cupdotop X_{4}$,
$I\cupdotop I$,
and
$X_{6}\cupdotop X_{6}$
has obstacle number at least $2$ and at most $3$.
The lower bounds are a consequence of Lemma~\ref{L:disjun},
using the previously determined outside obstacle numbers
for $X_{j}$, $j = 4,5,6$
(see Proposition~\ref{P:gesdp-obsout}).
The upper bounds are established by using the previously determined
outside obstacle embeddings for $X_{j}$ and embedding
two copies of $X_{j}$ using disjoint interior obstacles
(e.g., by placing the disjoint copies far apart).

However, we can show that each of these graphs
actually has obstacle number $2$.
Figure~\ref{F:twoObsDisjoint} shows a two-obstacle embedding of
$X_{4}\cupdotop X_{4}$,
found by grouping certain related vertices
in the $2$-obstacle embedding of $X_{4}$.
The other two graphs can be drawn similarly,
and we have the following.

\begin{figure}[htbp]
\begin{center}
\includegraphics[width=\linewidth]{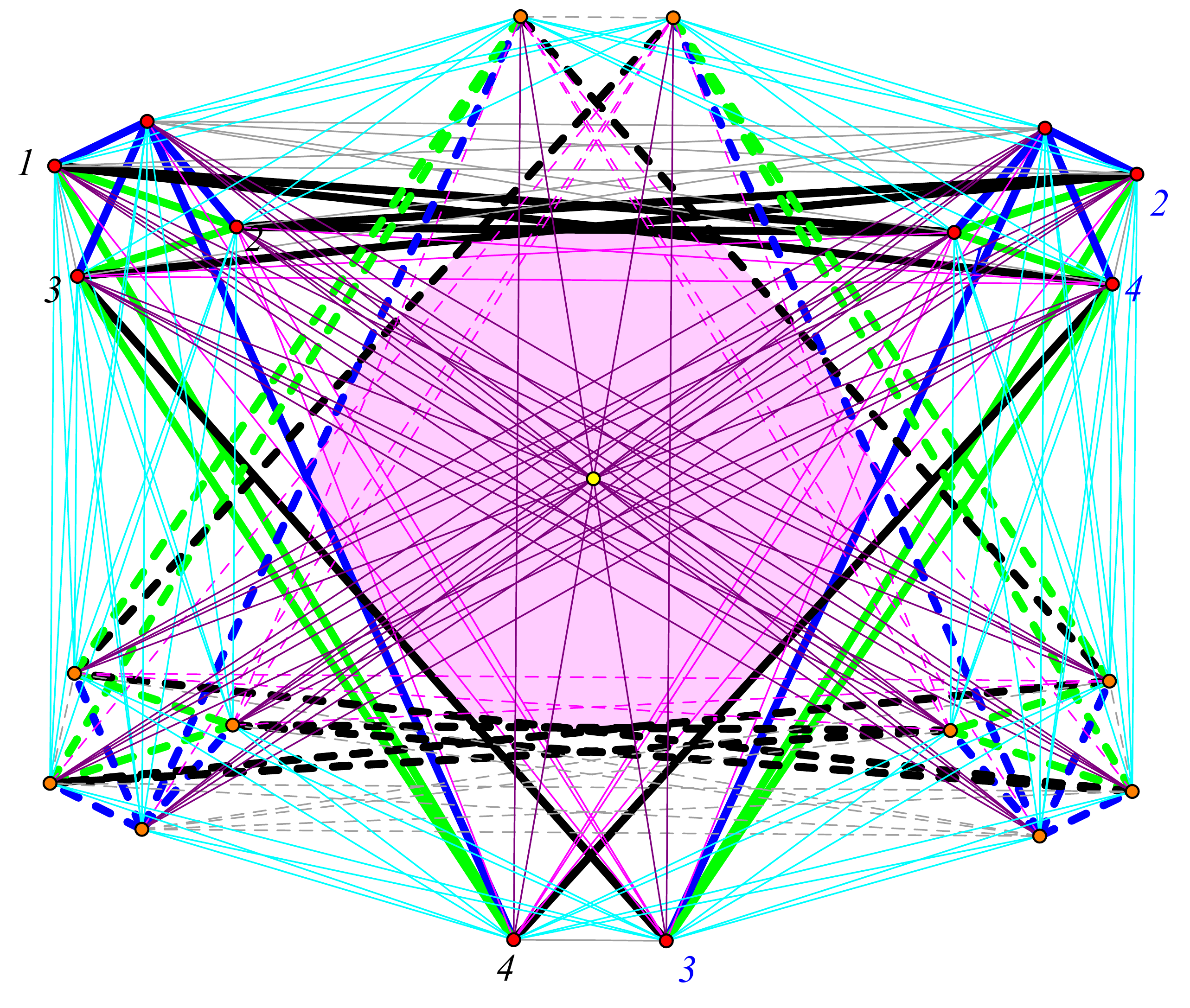}
\caption{A two-obstacle embedding of $X_{4}\cupdotop X_{4}$.
One copy of $X_{4}$ is shown with thick solid lines and red vertices;
the other copy is shown with thick dashed lines and orange vertices.
The interior obstacle is highlighted in pink.
Non-edges that are blocked by the interior obstacle that are
non-edges from the individual copies of $X_{4}$ are shown with magenta lines (dashed as appropriate),
while non-edges between the two copies that are blocked by the interior obstacle are drawn in purple.
Non-edges of the individual copies that are blocked by the outside obstacle
are drawn in gray (dashed and solid),
and non-edges between the two copies that are blocked by the outside obstacle are drawn in cyan.}
\label{F:twoObsDisjoint}
\end{center}
\end{figure}

\begin{corollary} \label{C:obs23}
All of the following hold.
\be
\item $\obs(X_{4}\cupdotop X_{4}) = 2$.
\item $\obs(I\cupdotop I) = 2$.
\item $ \obs(X_{6}\cupdotop X_{6}) = 2$.
\ee
\end{corollary}

We note that each of $X_{4}$, $I$, and $X_{6}$ is
the skeleton of a polyhedron, and thus is planar.
Further,
the disjoint union of planar graphs is planar;
each of the graphs in Corollary~\ref{C:obs23}
is a planar graph.
Thus
we have answered (negatively!) a question of
Alpert, Koch, \& Laison~\cite[p.~231]{AlpKocLai10},
asking whether every planar graph has obstacle number at most $1$. Note that $n(X_{4}\cupdotop X_{4}) = 20$, $n(I \cupdotop I) = 24$, and $n(X_{6}\cupdotop X_{6}) = 28$.

\begin{corollary}
There exists a planar graph with obstacle number greater than $1$.
\end{corollary}

\section{Obstacle Number, Part II} \label{S:obs2}

In this section we develop a variation on the SAT instance constructed
in Section~\ref{S:obsout}.
This new SAT instance will encode statements about arbitrary obstacles, instead of only outside obstacles. It will enable us to show directly that certain graphs have obstacle number greater than $1$.

In addition to
the $4$-Point Rule (Lemma~\ref{L:4pt})
and the $5$-Point Rule (Lemma~\ref{L:5pt}),
we use a new lemma, Lemma \ref{L:twoobs},
which plays role similar to that of Lemma~\ref{L:extobs}
in Section~\ref{S:obsout}.
For clarity in the next lemma, given distinct points $a$ and $b$, we denote the two closed halfplanes determined by line $\overleftrightarrow{ab}$ as $H^{+}_{ab}$ and $H^{-}_{ab}$, where  $H^{+}_{ab}$ contains all points $y$ such that either $y$ is on line $\overleftrightarrow{ab}$ or $aby$ is oriented clockwise.
An \emph{$ab$ key path with respect to $cd$}, denoted $P_{ab}(cd)$,
is a path from $a$ to $b$ that does not cross the line 
$\overleftrightarrow{cd}$;
that is, a path in $G$ from $a$ to $b$ that is entirely contained
in one of the closed halfplanes $H^{+}_{cd}$ or $H^{-}_{cd}$; see Figure \ref{fig:abKeyPath}.

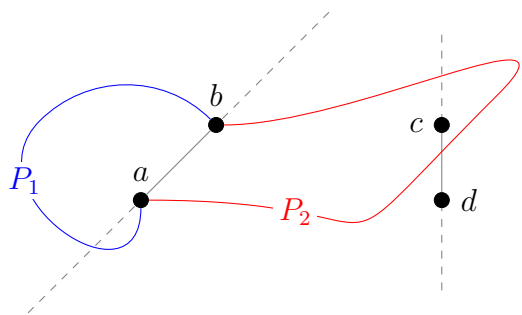
\begin{figure}[htbp]
\begin{center}

\begin{tikzpicture}
\node[draw, circle, fill=black, inner sep = 2pt, label=above:$a$] (A) at (0,0) {};
\node[draw, circle, fill=black,inner sep = 2pt, label=above:$b$] (B) at (1,1){};
\node[draw, circle,fill=black, inner sep = 2pt, label=left:$c$](C) at (4,1) {};
\node[draw, circle,fill=black, inner sep = 2pt, label=right:$d$](D) at (4, 0) {};

\draw[gray] (A) -- (B);
\draw[dashed, gray] ($(A)!-1.5!(B)$) -- (A);
\draw[dashed, gray] ($(B)!-1.5!(A)$) -- (B);

\draw[gray] (C) -- (D);
\draw[dashed, gray] ($(C)!-1.2!(D)$) -- (C);
\draw[dashed, gray] ($(D)!-1.2!(C)$) -- (D);

\node[] (p) at ($ (A)!1.2!90:(B)$) {} ;
\node[] (q) at ($ (D)!1.8!-30:(C)$) {};

\draw[blue] (A) to[ in=180+45, out=270,  looseness=2] node[near end, fill=white, inner sep = 2 pt] {$P_{1}$} (p) to[ in = 250, in = 90+45]  (B) ;

\draw[red, ] (A) to[ in=180+45, out=0,  looseness=2] node[fill=white, inner sep = 2 pt, near start] {$P_{2}$} (q) to[ in = 180+10, in = -0, looseness = 1.2]  (B) ;

\end{tikzpicture}

\caption{An $ab$ key path with respect to $cd$, denoted $P_{ab}(cd)$,
is a path from $a$ to $b$ that does not cross the line $\overset\leftrightarrow{cd}$. The path $P_{1}$ (blue) is an $ab$ key path with respect to $cd$, but the path $P_{2}$ (red) is not. 
}
\label{fig:abKeyPath}
\end{center}
\end{figure}

\begin{lemma} \label{L:twoobs}

Suppose we are given an obstacle representation of a graph $G$
that uses at most one obstacle.
Then, for each non-edge $ab$,
and for each non-edge $cd \neq ab$
($ab$ and $cd$ may share one vertex),
there exists a halfplane $H_{ab}(cd) \in \{H^{+}_{ab}, H^{-}_{ab}\}$
such that, for every $ab$ key path with respect to $cd$, $P_{ab}(cd)$,
some internal vertex of $P_{ab}(cd)$ lies in the interior of $H_{ab}(cd)$).

\end{lemma}

\begin{proof}

Choose any non-edge $ab$,
and then choose a second non-edge $cd$ distinct from $ab$.
Perturbing slightly if necessary,
we assume that $\overleftrightarrow{ab} \neq \overleftrightarrow{cd}$ as well
(see Observation~\ref{O:perturb}).
Suppose for a contradiction that there exist two distinct $ab$ paths
$P_{1}$ and $P_{2}$ in $G$,
both $ab$ key paths with respect to $cd$,
that lie in different closed halfplanes
determined by $\overleftrightarrow{ab}$.
Without loss of generality, we may assume that
$P_{1} \subseteq H^{+}_{ab}$ and $P_{2} \subseteq H^{-}_{ab}$.

Now, each of $P_{1}$, $P_{2}$ is contained in one of the two halfplanes
$H^{+}_{cd}$, $H^{-}_{cd}$,
because they are key paths with respect to $cd$.
Since $P_{1}$ and $P_{2}$ have common endpoints, namely $a$ and $b$,
we see that $P_{1}$ and $P_{2}$,
must be contained in the \emph{same} halfplane determined by
$\overleftrightarrow{cd}$. 

Therefore $P_{1}\cup P_{2}$ forms a closed path in the plane;
the open line segment $\overline{ab}$ lies in one component
of the complement of this closed path;
while the open line segment $\overline{cd}$ lies in a different component.
Since $G$ has only one obstacle, it is impossible for both
segments to be blocked,
a contradiction.

%
\end{proof}

Given a graph $G$,
we can use Lemmas~\ref{L:twoobs}, \ref{L:4pt} (the $4$-Point rule),
and \ref{L:5pt} (the $5$-Point Rule)
to create a SAT instance
encoding necessary conditions for the existence
of an obstacle representation of $G$
using at most $1$ obstacle.
Note that we are saying nothing here about outside obstacles.
Thus, if this SAT instance is not satisfiable,
then we may conclude that graph $G$ requires at least two obstacles,
that is, that $\obs(G) > 1$.

We represent triples of vertices
as we did before, in Section~\ref{S:obsout},
and
we encode the $4$-Point Rule and the $5$-Point Rule using the same clauses
as in our previous SAT instance.
We construct new SAT clauses based on Lemma~\ref{L:twoobs},
much as, in Section~\ref{S:obsout}, we constructed clauses based on
Lemma~\ref{L:extobs}.
We now look at how this is done.



Choose any non-edge $ab$ and choose any non-edge $cd$ distinct from $ab$. We illustrate the encoding of Lemma~\ref{L:twoobs} in terms of SAT clauses by showing how to encode statements about a particular path. Let $a,s,t,\dots,u,b$ be the sequence of vertices in some $a,b$-path $P$ (not necessarily a key path), where vertices $a$, $b$ are not adjacent (so that $ab$ is a non-edge). Let $c$, $d$ be nonadjacent vertices not lying on $P$. 

We introduce a new variable $k_{P(cd)}$ to represent the statement that
$P$ is an $ab$ key path with respect to $cd$.
If all vertices $v$ of $P$ produce triangles $cdv$
having the same orientation, then $P$ is a key path.
This is encoded by the following two clauses: 
\begin{gather}
\label{Cl:twoobs3}
\phantom{\neg} x_{cda} \vee \phantom{\neg} x_{cds} \vee
  \phantom{\neg} x_{cdt} \vee \dots \vee \phantom{\neg} x_{cdu}
  \vee \phantom{\neg} x_{cdb}
  \vee k_{P(cd)}\\
\label{Cl:twoobs4}
\neg x_{cda} \vee \neg x_{cds} \vee
  \neg x_{cdt} \vee \dots \vee \neg x_{cdu}
  \vee \neg x_{cdb}
  \vee k_{P(cd)}
\end{gather}

Next we encode
the statement that given this $ab$ and $cd$
we can find a special side of $ab$
so that every $ab$ key path with respect to $cd$
lies on the special side of $ab$.
The special side is encoded in another variable, $s_{ab, cd}$.
We canonically choose that $s_{ab, cd}$
represents the statement that the special halfplane is $H^{+}_{ab}$;
thus $\lnot s_{ab, cd}$
represents the statement that the special halfplane is $H^{-}_{ab}$.
Note that the special side of $ab$ depends on the choice of non-edge $cd$. 
The following clauses encode the desired statement:
\begin{gather}
\label{Cl:twoobs1}
%
\lnot k_{P(cd)} \lor  \lnot s_{ab, cd} \lor x_{abs} \lor x_{abt} \lor \cdots \lor x_{abu}\\ 
%
\label{Cl:twoobs2}
%
\lnot k_{P(cd)} \lor s_{ab,cd} \lor \lnot x_{abs} \lor \lnot x_{abt} \lor  \cdots \lor \lnot x_{abu} 
\end{gather}

\begin{observation} \label{O:sat-obs}
Let $G$ be a graph.
If $\obs(G) \le 1$,
then the SAT instance consisting of
all clauses
of the forms
(\ref{Cl:4pt}),
(\ref{Cl:5pt1}), and (\ref{Cl:5pt2}) from Section~\ref{S:obsout},
and
(\ref{Cl:twoobs3})--(\ref{Cl:twoobs2})
shown above---using canonical variables---is
satisfiable.
\end{observation}

Using these ideas,
we can determine the exact value of the obstacle number for
the icosahedron and some similar graphs.

\begin{proposition} \label{P:gesdp-obs}
All of the following hold.
\be
\item $\obs(X_{4}) = 2$.
\item $\obs(I)= \obs(X_{5}) = 2$.
\item $\obs(X_{6}) = 2$.
\ee
\end{proposition}

\begin{proof}
The lower bounds were found using a computer.
We create a SAT instance as described above,
using clauses representing the $4$-Point Rule,
the $5$-Point Rule,
and the statement of Lemma~\ref{L:twoobs}.
See~\cite{Cha17} for software to generate the SAT instances.
For each graph,
a standard SAT solver
(as before,
we used MiniSat \cite{MiniSat}, PicoSAT \cite{PicoSAT},
and zChaff \cite{zChaff})
indicates that
the SAT instance is not satisfiable.
Thus, by Observation~\ref{O:sat-obs},
we have $\obsout \ge 2$ for each graph.

For the upper bound, we apply Proposition~\ref{P:gesdp-obsout},
which states that each graph considered here has $\obsout = 2$,
and note that $\obs \le \obsout$ for every graph.
\end{proof}

Note that
Part (2) of Proposition~\ref{P:gesdp-obs}
answers a question of
Alpert, Koch, \& Laison~\cite[p.~231]{AlpKocLai10}
asking for the obstacle number of the icosahedron.

Note also that the graph $X_{4}$,
mentioned in part (1) of Proposition~\ref{P:gesdp-obs},
has order $10$.
This is thus the second known example of a graph of order $10$
with obstacle number $2$
(the first being $K^*_{5,5}$,
shown to have obstacle number $2$ by
Pach \& Sari\"{o}z~\cite[Thm.~2.1]{PacSar11}).
But unlike the Pach-Sari\"{o}z example, graph $X_4$ is planar.
We do not know whether there is any planar graph---or,
indeed, any graph at all---of
smaller order
that has obstacle number $2$.

Alpert, Koch, \& Laison~\cite[p.~231]{AlpKocLai10}
asked for the obstacle number of the dodecahedron.
We answer this question as follows.
\begin{proposition} \label{dodec-obs}
Let $D$ be the dodecahedron.
Then $\obs(D) = \obsout(D) = 1$.
\end{proposition}
\begin{proof}
Figure~\ref{F:dodec}
shows an obstacle representation of the dodecahedron,
using a single outside obstacle.
\end{proof}

\begin{figure}[htbp]
\begin{center}
\ffigbox[]{
\begin{subfloatrow}[2]
\ffigbox{\caption{A drawing of the dodecahedron}}{\includegraphics[width = \linewidth]{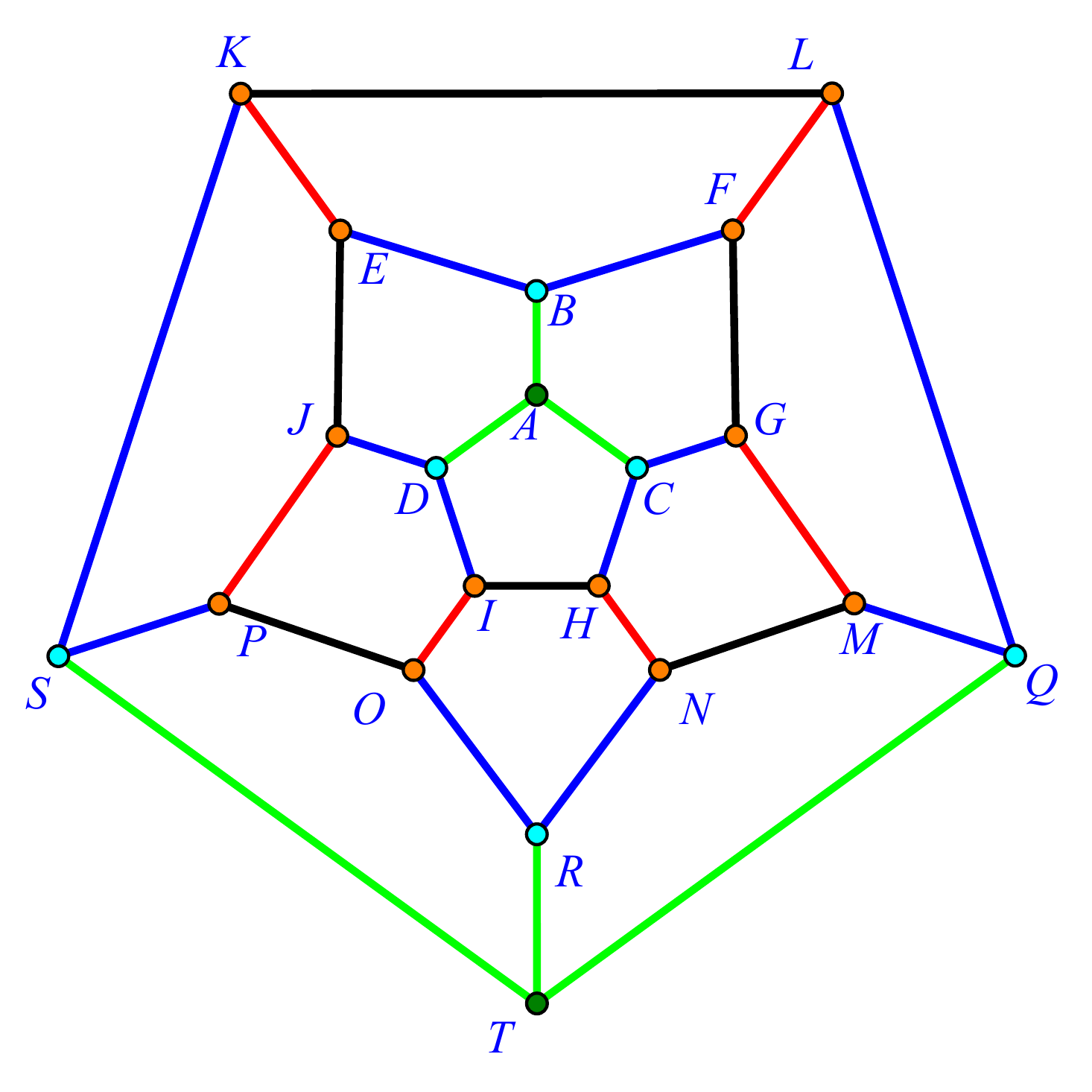}
}
\ffigbox{\caption{An obstacle representation using a single outside obstacle.}}{\includegraphics[width = \linewidth]{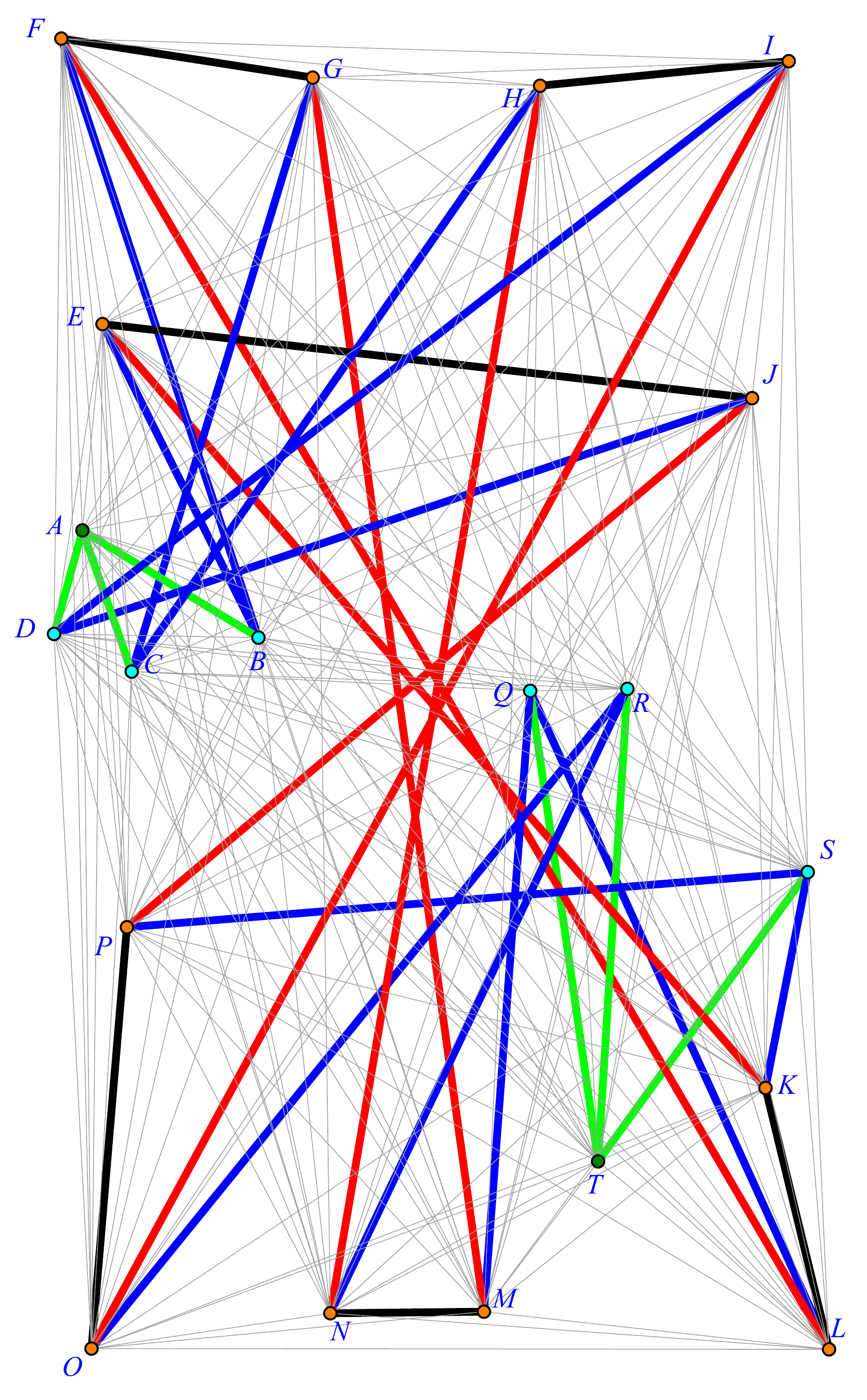}
}
\end{subfloatrow}
}
{\caption{A drawing of the dodecahedron, with useful edge-colorings, and an obstacle representation of the dodecahedron using a single outside obstacle (with corresponding edges and vertices).}
\label{F:dodec}
}
\end{center}
\end{figure}

%
%
%
%

\section{Open questions}

There are several interesting open questions related to obstacle numbers of graphs with small numbers of vertices.

\begin{question}
What is the minimum order of a planar graph
with obstacle number $2$?
\end{question}

By Proposition~\ref{P:ord5},
the above minimum must be at least $6$.

We have not been able to show that there exists
a planar graph with obstacle number greater than $2$.
It is natural to ask whether there exists an upper bound
on the obstacle numbers of planar graphs,
and, if so, what it is.
It seems likely that either there is no such upper bound,
or else the maximum obstacle number of a planar graph
is $2$.
We conjecture that the latter option holds.

\begin{conjecture} \label{J:obsplanar2}
If $G$ is a planar graph,
then $\obs(G)\le 2$.
\end{conjecture}

%
%
%

In general,
little is known about the least order
of a graph with any particular
(outside) obstacle number.

\begin{question} What is the smallest order
of a
graph $G$ with $\obsout(G) = 2$?
With $\obs(G) = 2$?
With $\obsout(G)$ or $\obs(G) > 2$?\end{question}

Again, by Proposition~\ref{P:ord5},
in each case the answer
must be at least $6$.

We have found the above questions quite
resistant to solution. Perhaps this is unsurprising since Johnson and Sarioz showed in \cite{JohSar11} that computing the obstacle number of a plane graph is NP-hard.
Two examples of graphs of order $10$
with obstacle number $2$ have been found,
namely $K_{5,5}^{*}$ and $X_{4}$;
none of smaller order
are known.
If we knew of a single graph of order $9$ or less
for which one of the SAT instances we construct
is not satisfiable,
then we could reduce our current bound of $10$;
however we have found no such graph.

It seems plausible that an approach to answering
the above questions
would be a brute-force application of SAT instances
to all graphs with order strictly less than $10$.
However, this naive approach has two flaws.
First,
there are a large number of graphs of
order at most $9$
(for example, there are $11117$ connected graphs of order $8$
and $261080$ connected graphs of order $9$~\cite[Sequence~A001349]{OEIS}),
and there are significant time and computational issues
involved in processing
the SAT instances for all these graphs.

Second,
while non-satisfiability of the SAT instance
for one of these graphs would imply
that the corresponding (outside) obstacle number
was strictly greater than $1$,
satisfiability of
an instance
does not imply any bound on
the corresponding obstacle number.
The solution of one of our SAT instances
gives only a specification of clockwise/counter-clockwise
orientation
for each triple of points.
This might not correspond to any actual point placement in the plane.
Or it may correspond to many point placements.
And even if one of these gives the desired obstacle representation,
others may not;
or none of them may.
Furthermore, a single SAT instance
can have exponentially many solutions,
each of which may need to be checked,
in order to find an obstacle representation.
In any case,
satisfiability provides us
only with a starting point in
the search for an obstacle representation;
we know of no efficient, reliable technique
for actually finding such a representation without human intervention
and invention.

\bibliographystyle{amsplain}
\bibliography{ObstacleNumbers}

\providecommand{\bysame}{\leavevmode\hbox to3em{\hrulefill}\thinspace}
\providecommand{\MR}{\relax\ifhmode\unskip\space\fi MR }
\providecommand{\MRhref}[2]{%
  \href{http://www.ams.org/mathscinet-getitem?mr=#1}{#2}
}
\providecommand{\href}[2]{#2}
\begin{thebibliography}{10}

\bibitem{OEIS}
\emph{The {O}n-{L}ine {E}ncyclopedia of {I}nteger {S}equences}, 2013,
  \url{http://oeis.org}.

\bibitem{AlpKocLai10}
Hannah Alpert, Christina Koch, and Joshua~D. Laison, \emph{Obstacle numbers of
  graphs}, Discrete Comput. Geom. \textbf{44} (2010), no.~1, 223--244.
  \MR{2639825 (2011g:05208)}

\bibitem{PicoSAT}
A.~Biere, \emph{Pico{SAT} essentials}, Journal on Satisfiability, Boolean
  Modeling and Computation \textbf{4} (2008), 75--97.

\bibitem{BjoLasStu99}
Anders Bj{{\"o}}rner, Michel Las~Vergnas, Bernd Sturmfels, Neil White, and
  G{{\"u}}nter~M. Ziegler, \emph{Oriented matroids}, 2nd ed., Encyclopedia of
  Mathematics and its Applications, vol.~46, Cambridge University Press,
  Cambridge, 1999. \MR{1744046 (2000j:52016)}

\bibitem{BonMur76}
J.~A. Bondy and U.~S.~R. Murty, \emph{Graph theory with applications}, American
  Elsevier Publishing Co., Inc., New York, 1976. \MR{0411988 (54 \#117)}

\bibitem{Cha17}
Glenn~G. Chappell, \emph{Software for computing sat instances related to obstacle
  number}.

\bibitem{DumPacTot09}
Adrian Dumitrescu, J{{\'a}}nos Pach, and G{{\'e}}za T{{\'o}}th, \emph{A note on
  blocking visibility between points}, Geombinatorics \textbf{19} (2009),
  no.~2, 67--73. \MR{2571951}

\bibitem{FulSaeSar2013}
Radoslav Fulek, Noushin Saeedi, and Deniz Sar{\i}{{\"o}}z, \emph{Convex
  obstacle numbers of outerplanar graphs and bipartite permutation graphs},
  Thirty essays on geometric graph theory, Springer, New York, 2013,
  pp.~249--261. \MR{3205157}

\bibitem{Gar}
{O}pen~{P}roblem {G}arden, \emph{Obstacle number of planar graphs},
  \url{http://www.openproblemgarden.org/op/obstacle_number_of_planar_graphs}.

\bibitem{JohSar11}
Matthew~P. Johnson and Deniz Sarioz, \emph{Computing the obstacle number of a
  plane graph}, 08 2011, \url{http://arxiv.org/abs/1107.4624v2}.

\bibitem{Joh66}
Norman~W. Johnson, \emph{Convex polyhedra with regular faces}, Canad. J. Math.
  \textbf{18} (1966), 169--200. \MR{0185507}

\bibitem{Knu92}
D.~E. Knuth, \emph{Axioms and hulls}, Lecture Notes in Computer Science, vol.
  606, Springer-Verlag, Berlin, 1992. \MR{1226891}

\bibitem{Lai}
Josh Laison, \emph{Obstacle numbers of graphs}, talk at {C}ombinatoral
  {P}otlach, Seattle University, 2011.

\bibitem{Mat09}
Ji{\v{r}}{\'{\i}} Matou{\v{s}}ek, \emph{Blocking visibility for points in
  general position}, Discrete Comput. Geom. \textbf{42} (2009), no.~2,
  219--223. \MR{2519877 (2010f:52027)}

\bibitem{MiniSat}
Mini{S}at, \url{http://minisat.se}.

\bibitem{MukPacPal12}
Padmini Mukkamala, J{{\'a}}nos Pach, and D{{\"o}}m{{\"o}}t{{\"o}}r
  P{{\'a}}lv{{\"o}}lgyi, \emph{Lower bounds on the obstacle number of graphs},
  Electron. J. Combin. \textbf{19} (2012), no.~2, Paper 32, 8. \MR{2928647}

\bibitem{MukPacSar10}
Padmini Mukkamala, J{{\'a}}nos Pach, and Deniz Sar{\i}{{\"o}}z, \emph{Graphs
  with large obstacle numbers}, Graph-theoretic concepts in computer science,
  Lecture Notes in Comput. Sci., vol. 6410, Springer, Berlin, 2010,
  pp.~292--303. \MR{2765279 (2012e:05100)}

\bibitem{PacSar10}
J{\'a}nos Pach and Deniz Sar{\i}{{\"o}}z, \emph{Small (2,s)-colorable graphs
  without 1-obstacle representations}, ArXiV, http://arxiv.org/abs/1012.5907
  (2010), no.~Ancillary to ``On the structure of graphs with low obstacle
  number", Graphs and Combinatorics, Volume 27, Number 3.

\bibitem{PacSar11}
J{{\'a}}nos Pach and Deniz Sar{\i}{{\"o}}z, \emph{On the structure of graphs
  with low obstacle number}, Graphs Combin. \textbf{27} (2011), no.~3,
  465--473. \MR{2787431 (2012d:05268)}

\bibitem{zChaff}
{zC}haff, \url{http://www.princeton.edu/~chaff/zchaff.html}.

\end{thebibliography}

\end{document}